\documentclass[a4paper,12pt,leqno,oneside]{amsart}
\usepackage{amssymb}
\usepackage{mathrsfs}
\usepackage{graphicx,color}
\usepackage{newtxtext}
\usepackage[varg]{newtxmath}
\usepackage[shortlabels]{enumitem}
\usepackage{geometry}%
\theoremstyle{plain} 
\newtheorem{theorem}{Theorem}[section]
\newtheorem{corollary}[theorem]{Corollary}
\newtheorem{proposition}[theorem]{Proposition}
\newtheorem{lemma}[theorem]{Lemma}

\theoremstyle{definition} 
\newtheorem{definition}[theorem]{Definition}
\newtheorem{remark}[theorem]{Remark}

\newcommand{\R}{\mathbb{R}}

\usepackage{constants}
\newconstantfamily{eps}{symbol=\varepsilon}
\numberwithin{equation}{section}
\usepackage{lineno}

\usepackage{bm}
\renewcommand{\vec}[1]{\bm{#1}}
\usepackage{hyperref}

\title{Local well-posedness of a nonlinear Fokker-Planck model}

\author{Yekaterina Epshteyn}
\address[Yekaterina Epshteyn]%
{Department of Mathematics,
The University of Utah,
Salt Lake City, UT 84112, USA}
\email{epshteyn@math.utah.edu}

\author{Chang Liu}
\address[Chang Liu]%
{Department of Mathematics,
The University of Utah,
Salt Lake City, UT 84112, USA}
\email{liukamala@math.utah.edu}

\author{Chun Liu}
\address[Chun Liu]%
{Department of Applied Mathematics, Illinois Institute of Technology.
Chicago, IL 60616, USA}
\email{cliu124@iit.edu}

\author{Masashi Mizuno}
\address[Masashi Mizuno]%
{Department of Mathematics, College of Science
and Technology, Nihon University, Tokyo 101-8308 JAPAN}
\email{mizuno.masashi@nihon-u.ac.jp}


\allowdisplaybreaks[1]
\begin{document}

%
%


%
%
%

\begin{abstract}
Noise or fluctuations play an important role in the modeling and understanding of the behavior
of various complex systems in nature. Fokker-Planck equations are
powerful mathematical tools to study behavior of such systems subjected
to fluctuations. In this paper we establish local well-posedness 
result of a new nonlinear Fokker-Planck equation. Such equations
appear in the modeling of the grain boundary dynamics during
microstructure evolution in the polycrystalline materials and obey special energy laws.
\end{abstract}

\maketitle

\section{Introduction}
\label{sec:1}
\par Fluctuations play an essential role in the modeling and understanding of the behavior
of various complex processes. Many natural systems are affected by different
external and internal mechanisms that are not known explicitly, and
very often described as fluctuations or noise. Fokker-Planck models
are widely used
as a versatile mathematical tool to describe the macroscopic behavior of the systems
that undergo such fluctuations, see more detailed discussion and
examples in \cite{MR987631,MR2053476,MR3932086,doi:10.1137/S0036141096303359,MR3019444,MR3485127,MR4196904,MR4218540}, among
many others. In our previous work we derived Fokker-Planck type
systems as a part of grain growth models of polycrystalline materials,
e.g. \cite{DK:gbphysrev,MR2772123,MR3729587,epshteyn2021stochastic}. 

From the thermodynamical point of view, many Fokker-Planck type systems can be viewed as special cases of
general diffusion \cite{GiKiLi16}. They can be derived from the kinematic continuity equations, the conservation law, and the specific energy
dissipation law, using the energetic variational approaches \cite{onsager1931reciprocal2,GiKiLi16}.
We want to point out that while the linear and nonlinear
Fokker-Planck models with the energy laws can be obtained using such
energetic variational approach, not all Fokker-Planck systems
derived from stochastic differential equations (SDEs) by the Ito process 
have underlying energy law principles \cite{risken1996fokker}.

First,  consider the following conservation law subject to the natural boundary condition,
\begin{equation}
\label{eq:4-1}
 \left\{
  \begin{aligned}
   \frac{\partial f}{\partial t}
   +
   \nabla \cdot
   (f\vec{u})
   &=
   0,&\quad
   &t>0,\ x\in\Omega, \\
   f\vec{u}\cdot\nu|_{\partial\Omega}
   &=
   0,&\quad
   &t>0.
  \end{aligned}
 \right.
\end{equation}
Here $\Omega\subset\R^n$ is a convex domain,
$f=f(x,t):\Omega\times[0,T)\rightarrow\R$ is a probability density
function, $\vec{u}$ is the velocity vector which 
depends on $x$, $t$, and the probability density function $f$, and $\nu$
is an outer unit normal to the boundary $\partial\Omega$ of the domain
$\Omega$. We assume that the above system \eqref{eq:4-1} also satisfies
the following energy law,
\begin{equation}
 \label{eq:4-2}
  \frac{d}{dt}
  \int_\Omega \omega(f,x)\,dx
  =
  -
  \int\pi(f,x,t)
  |\vec{u}|^2\,dx.
\end{equation}
Here,  $\omega=\omega(f,x)$ represents the free energy,
  which defines the equilibrium state of the
system, and $\pi(f,x,t)$ 
is the so-called mobility function which defines the
  evolution of the system to the  equilibrium state. The specific forms of these quantities
will be
discussed in more details below. Now, take a formal time-derivative on the
left-hand side of \eqref{eq:4-2}, then using integration by parts
together with system \eqref{eq:4-1}, we get,
\begin{equation}
 \label{eq:4-3}
  \begin{split}
   \frac{d}{dt}
   \int_\Omega \omega(f,x)\,dx
   &=
   \int_\Omega \omega_f(f,x)f_t\,dx
   \\
   &=
   -\int_\Omega \omega_f(f,x)\nabla\cdot(f\vec{u})\,dx
   =  
   \int_\Omega \nabla\omega_f(f,x)\cdot(f\vec{u})\,dx.
  \end{split}
\end{equation}
Using relations \eqref{eq:4-2} and \eqref{eq:4-3}, we have that,
\begin{equation*}
 -
  \int\pi(f,x,t)
  |\vec{u}|^2\,dx
  =
  \int_\Omega \nabla\omega_f(f,x)\cdot(f\vec{u})\,dx.
\end{equation*}
Thus, the velocity field $\vec{u}$ of the model
\eqref{eq:4-1}-\eqref{eq:4-2} should satisfy the following relation,
\begin{equation}
 \label{eq:4-4}
  -
  \pi(f,x,t)
  \vec{u}
  =
  f
  \nabla(\omega_f(f,x)).
\end{equation}
In fact \eqref{eq:4-4} represents the force balance equation for the system. 
The left hand side represents the dissipative force and the right hand
side is the conservative force obtained using the
free energy of the system. This derivation is  consistent with
the general energetic variational approach in \cite{onsager1931reciprocal2,GiKiLi16}.

Let us put this discussion in the context of linear and nonlinear
Fokker-Planck models now.

 %
 
  Such systems arise in many physical and engineering
   applications, e.g.,  \cite{coleman1967thermodynamics, dafermos1978second,
     DK:gbphysrev,MR2772123,MR3729587,epshteyn2021stochastic,liu2022brinkman}. One
 example of the application of Fokker-Planck systems is the modeling
 of grain growth in polycrystalline materials. Many technologically
 useful materials appear as polycrystalline microstructures, composed
 of small monocrystalline cells or grains, separated by
 interfaces, or grain boundaries of crystallites with different
 lattice orientations. In a planar grain boundary network, a point
 where three grain
boundaries meet is called a triple junction point, see
Fig. ~\ref{figTJ}. Grain growth is a very complex multiscale and
multiphysics process influenced by the dynamics of grain boundaries,
triple junctions and the dynamics of lattice misorientations
(difference in the lattice orientations between two neighboring grains that
share the grain boundary, Fig. ~\ref{figTJ}), e.g.,
\cite{Katya-Chun-Mzn4,PATRICK2023118476,paperRickman}. In case of the grain growth
modeling \cite{epshteyn2021stochastic}, in the Fokker-Planck system, $f$ 
 may describe the joint distribution function of the lattice misorientation of the
 grain boundaries and of the position of the triple junctions, $\phi$ may
 describe the grain boundary energy density, and $D$ is
 related to the absolute temperature of the entire system
 \cite{lai2022positivity} (it can be viewed as a function of
  the fluctuation parameters of the lattice misorientations and of the
 position of the triple junctions due to fluctuation-dissipation
 principle \cite{epshteyn2021stochastic}).

\begin{figure}[h]
 \centering \includegraphics[width=8cm]{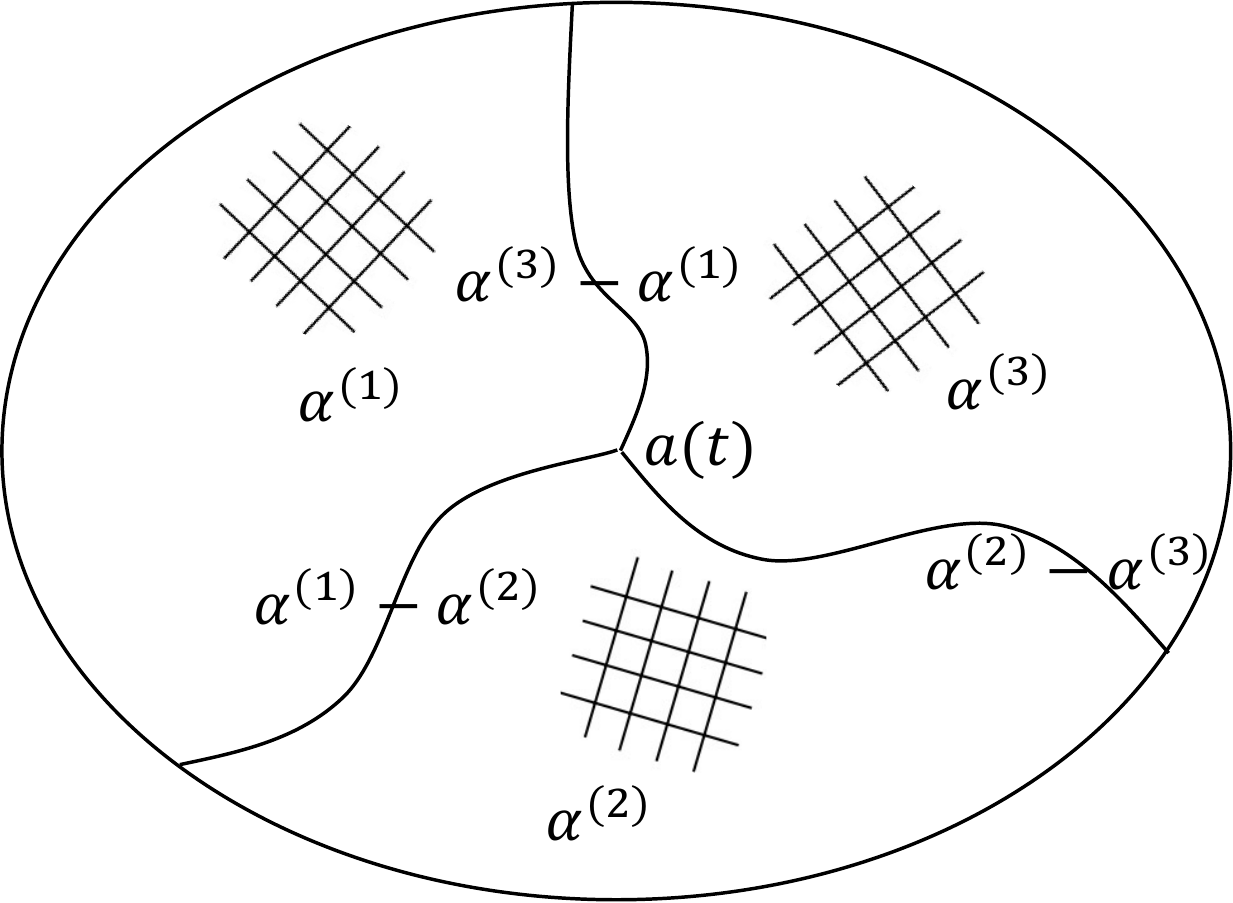}
 \caption{Illustration of the three grain boundaries that meet at a triple
 junction which is positioned at the  $\vec{a}(t)$. Each grain boundary has a lattice
 misorientation which
 is the difference between lattice (lined grids on the figure) orientations $\alpha^{(j)}, j=1, 2,
 3$ of the grains that share the grain boundary. In
 \cite{epshteyn2021stochastic}, a grain boundary network was considered
 as a system of such triple junctions and the grain boundaries
 misorientations, and was modeled by the Fokker-Planck
   equation for the joint
   distribution function of
 the position of the triple junctions and the misorientations.}
 \label{figTJ}
\end{figure}

In the cases when
 $\omega(f,x)=Df(\log f-1)+f\phi$ (free energy density) and
 $\pi(f,x,t)=f(x,t)$ (mobility), where $D>0$
 is a positive constant and the potential function $\phi=\phi(x)$ is a given function. $D$ being a
 constant  is the case of the system with homogeneous
 absolute temperature \cite{coleman1967thermodynamics,ericksen1998introduction}.  We will recover the corresponding linear
 Fokker-Planck model from  conservation and energy laws,
 \eqref{eq:4-1}-\eqref{eq:4-2}. 
 First, the direct computation yields, 
 \begin{equation*}
   f\nabla \omega_f
   =
   f
   \nabla(D\log f+\phi(x)).
 \end{equation*}
 Hence,  from \eqref{eq:4-4}, the velocity field $\vec{u}$ should be, 
 \begin{equation}
  \label{eq:4-5}
   \vec{u}
   =
   -
   \nabla(D\log f+\phi(x))=-\Big(D\frac{\nabla f}{f}+\nabla \phi(x)\Big).
 \end{equation}
Using vector field \eqref{eq:4-5} in the conservation law \eqref{eq:4-1}, we
obtain the following \emph{linear} Fokker-Planck equation,
 \begin{equation}
  \label{eq:4-6}
   \frac{\partial f}{\partial t}
   =\nabla\cdot(\nabla \phi(x) f)+\nabla\cdot(D\nabla f).   
 \end{equation}
Note, that the linear Fokker-Planck equation has the associated
Langevin equation \cite{risken1996fokker,gardiner1985handbook},
\begin{equation}
  \label{eq:4-6l}
  dx=-\nabla\phi(x)dt+\sqrt{2D}dB.
 \end{equation}
The linear Fokker-Planck equation \eqref{eq:4-6} can also be derived from
the corresponding
Langevin equation \eqref{eq:4-6l} (see \cite{MR3932086}).

 Some diffusion equations  can be interpreted using the idea of Brownian motion \cite{gardiner1985handbook}. Consider random process
\begin{equation}
dx =    \upsilon (x) dt + \sigma (x)  dB,
\end{equation}
where $B$ is standard Brownian motion. With a Taylor expansion of
probability density function $f (x, t)$, one can obtain the following PDEs:
\begin{itemize}
\item  Ito calculus provides, $f_t + \nabla \cdot(\upsilon f)= \frac{1}{2} \Delta(\sigma^2 f) $.
\item The derivation using Stratonovich integral yields, $f_t + \nabla\cdot (\upsilon f ) =\frac{1}{2}  \nabla\cdot [\sigma \nabla (\sigma f )]$.
\item One can also derive PDE with self-adjoint diffusion term,
  namely,  $ f_t + \nabla\cdot (\upsilon f ) =\frac{1}{2}  \nabla\cdot [\sigma^2 \nabla (f )]$. 
 \end{itemize}
 In many cases, these models can also be treated in the general framework of energetic variational approach.
Following the fluctuation-dissipation theorem
\cite{de2013non,kubo1966fluctuation}, taking the convection coefficient,
$\upsilon (x) = -\frac{1}{2} \sigma (x) ^2 \nabla \phi$,
and assuming that $f$ satisfies the conservation law $f_t +\nabla\cdot( u f) = 0$, the equations above 
satisfy and can also be obtained from variation of the following
energy laws \cite{GiKiLi16},
\begin{itemize} 
 \item For Ito, $\frac{d}{dt} \int_{\Omega} [ f\ln (\frac{1}{2}\sigma^2 f) + \phi f ] \, dx = - \int_{\Omega} \frac{f}{\frac{1}{2}\sigma^2 } |u|^2 \, dx.$
  \item For Stratonovich, $\frac{d}{dt} \int_{\Omega} [ f\ln (\sigma f) + \phi f ] \, dx = - \int_{\Omega} \frac{f}{\frac{1}{2}\sigma^2 } |u|^2 \, dx.$
  \item For self-adjoint case,  $\frac{d}{dt} \int_{\Omega} [ f\ln f+ \phi f ] \, dx = - \int_{\Omega} \frac{f}{\frac{1}{2}\sigma^2 } |u|^2 \, dx,$
 \end{itemize}
 where $\Omega\subset\R^d$ is a bounded domain, $d\ge 1$.

 In this paper, instead of starting from the stochastic differential equations, we will derive the system from the energetic aspects,
 by prescribing the kinematic conservation law and the energy dissipation law.
 We will  consider the case of the inhomogeneous absolute
 temperature and more general dissipation mechanism. In
   particular, we look at the case with
 $\omega(f,x)=D(x)f(\log f-1)+f\phi(x)$, 
 and
 $\pi(f,x,t)=2D(x)f/(b(x,t))^2$,  where
 $D=D(x)$ and $\phi=\phi(x)$ are positive functions. The function $b(x,t)$ is also 
 positive, and provides the extra freedom in the dissipation mechanism.
As discussed above, such
 systems may arise in the grain growth modeling, e.g. \cite{epshteyn2021stochastic,Katya-Kamala-Chun-Masashi}.
 In particular, the temperature, in terms of $D$ in this context, will account for some information of the under-resolved mechanisms in the systems,
 such as critical events/disappearance events (e.g. grain disappearance, facet/grain boundary disappearance, facet
interchange, splitting of unstable junctions and nucleation of the
grains). The specific form of the mobility function $\pi(f, x, t)$ here is the direct
 consequence of the fluctuation-dissipation theorem
 \cite{kubo1966fluctuation,de2013non,epshteyn2021stochastic}, which
 ensures that the system under consideration will approach the equilibrium configuration.

 Since, in this case, the conservative force takes the form
 \begin{equation*}
  f\nabla \omega_f
   =
   f
   \nabla(D(x)\log f+\phi(x)).
 \end{equation*}
 Hence, from \eqref{eq:4-4}, the velocity field $\vec{u}$ will be,
 \begin{equation}
  \label{eq:2-0-5}
   \vec{u}
   =
 -
  \frac{(b(x,t))^2}{2D(x)}
  \nabla(D(x)\log f+\phi(x)).
 \end{equation}
Using formula \eqref{eq:2-0-5} in the conservation law \eqref{eq:4-1}, we
obtain the \emph{nonlinear} Fokker-Planck equation (with energy law
as defined in \eqref{eq:4-2}, see also discussion below in Section~\ref{sec:2}),
\begin{equation}
\label{eq:4-8}
 \frac{\partial f}{\partial t}
  -
  \nabla \cdot
  \left(
   \frac{(b(x,t))^2}{2D(x)}
   f\nabla(D(x)\log f+\phi(x))
  \right)
  =
  0.   
\end{equation}
Note, that the nonlinearity $f\log f$ in \eqref{eq:4-8} comes as a
result of inhomogeneity of the absolute
 temperature $D(x)$. In addition, in contrast with the linear
 Fokker-Planck model \eqref{eq:4-6}, the nonlinear Fokker-Planck model
 does not have the corresponding Langevin equation.  Instead it has the
 associated stochastic differential equation with coefficients that
 depend on the probability density $f(x, t)$.

\par This work establishes
 local well-posedness of the new nonlinear Fokker-Planck type model
 \eqref{eq:4-8} subject to the boundary and initial conditions. Note, inhomogeneity and resulting non-linearity in
 the new model \eqref{eq:4-8} are very
different from the vast existing literature on the Fokker-Planck type models. They come as a result of inhomogeneous
absolute temperature in a free energy for the system
\eqref{eq:2-0-12}. Such absolute temperature gives rise to a nonstandard nonlinearity of the
form $f\nabla D(x) \log f$ in the corresponding PDE model
(see \eqref{eq:4-8}, or \eqref{eq:2-0-4} in Section~\ref{sec:2} below).
For example, any conventional entropy methods, including Bakry-Emory
method \cite{MR3497125} do not extend
to such models in a standard or trivial way. In addition models like
\eqref{eq:4-8} or \eqref{eq:2-0-4}  appear as subsystems in the much
more complex systems in the grain growth modeling in polycrystalline
materials,  and hence one needs to know properties of the classical
solutions to such PDEs. 

\par The paper is organized as follows. In Section~\ref{sec:2}, we first
state the nonlinear Fokker-Planck system and validate energy law using
given partial differential equation and the boundary conditions. After
that we show local existence of the solution to the model. In Section~\ref{sec:3},
we establish uniqueness of the local solution. Some conclusions are
given in Section~\ref{sec:4}.

\section{Existence of a local solution}
\label{sec:2}

In this section, we will provide a constructive proof of the existence
of a local classical solution of the following nonlinear Fokker-Planck
type equation with the natural boundary condition (see also
\eqref{eq:4-8} in Section~\ref{sec:1}):
\begin{equation}
 \label{eq:2-0-4}
  \left\{
   \begin{aligned}
    &\frac{\partial f}{\partial t}
    =
    -
    \nabla\cdot
    \left(
    \left(
    -
    \frac{(b(x,t))^2}{2D(x)}\nabla\phi(x)
    -
    \frac{(b(x,t))^2}{2D(x)}\log f\nabla D(x)
    \right)
    f
    \right)
    +
    \frac{1}{2}
    \nabla\cdot
    ((b(x,t))^2\nabla f),
    \quad
    x\in\Omega,\ t>0, \\
    &\left(
    \frac{(b(x,t))^2}{2D(x)}f\nabla\phi(x)
    +
    \frac{(b(x,t))^2}{2D(x)}f\log f\nabla D(x)
    +
    \frac12
    (b(x,t))^2\nabla f
    \right)
    \cdot
    \nu
    \bigg|_{\partial\Omega}
    =
    0, 
    \quad
    t>0,\\
    &f(x,0)
    =
    f_0(x),\quad
    x\in\Omega, 
   \end{aligned}
  \right.
\end{equation}
where $\Omega\subset\R^d$ is a bounded domain, $d\ge 1$. Here $b=b(x,t)$
is a positive function on $\Omega\times[0,\infty)$, $D=D(x)$ is a
positive function on $\Omega$, $f_0=f_0(x)$ is a
suitable (to be defined later through $\rho_0$ in
\eqref{eq:2-0-2} and \eqref{eq:2-0-3}) positive probability density function on $\Omega$ and
$\phi=\phi(x)$ is a function on $\Omega$. A function $f=f(x,t)>0$ is an
unknown probability density function.

The Fokker-Planck equation \eqref{eq:2-0-4} has a dissipative structure
for the following free energy,
\begin{equation}
 \label{eq:2-0-12}
  F[f]
  :=
  \int_{\Omega}
  \left(
  D(x)f(x,t)(\log f(x,t)-1)
  +
  f(x,t)\phi(x)
  \right)
  \,dx.
\end{equation}
Below, we validate an energy law for the Fokker-Planck equation
\eqref{eq:2-0-4} by performing formal calculations.

\begin{proposition}
 Let $b=b(x,t)$, $D=D(x)$, $f_0=f_0(x)$, $\phi=\phi(x)$ be
 sufficiently smooth
 functions. Then a classical solution $f$ of the Fokker-Planck equation
 \eqref{eq:2-0-4} satisfies the following energy law,
 \begin{equation}
  \label{eq:2-0-8}
   \frac{d}{dt}F[f]
   =
   -
   \int_\Omega
   \frac{(b(x,t))^2}{2D(x)}
   \left|
    \nabla
    (
    \phi(x)
    +
     D(x)\log f(x,t)
     )
   \right|^2
   f(x,t)
   \,dx.
 \end{equation}
\end{proposition}

\begin{proof}
 Here, we will validate the energy law via calculation of the rate of change
 of the free energy $F$ (see also relevant discussion in Section~\ref{sec:1} where we
 postulated the energy law for the model and derived the velocity
 field, and hence the PDE as
 a consequence). By direct computation of
 $\frac{dF}{dt}$ and using the Fokker-Planck equation \eqref{eq:2-0-4} together
 with $\nabla f=f\nabla\log f$, we have,
 \begin{equation}
  \label{eq:2-0-6}
  \begin{split}
   \frac{d}{dt}F[f]
   &=
   \int_\Omega
   \left(
   D(x)\log f(x,t)
   +
   \phi(x)
   \right)
   \frac{\partial f}{\partial t}(x,t)
   \,dx \\
   &=
   -
   \int_\Omega
   \left(
   D(x)\log f(x,t)
   +
   \phi(x)
   \right)
   \nabla\cdot(f(x,t)\vec{u})
   \,dx,
  \end{split}
 \end{equation}
 where we introduced the velocity vector field as,
 \begin{equation}
  \label{eq:2-0-13}
   \vec{u}
   :=
   -
   \frac{(b(x,t))^2}{2D(x)}\nabla\phi(x)
   -
   \frac{(b(x,t))^2}{2D(x)}\log f(x,t)\nabla D(x)
   -
   \frac12
   (b(x,t))^2\nabla \log f(x,t).
 \end{equation}
 Note that, $\nabla (D(x)\log f(x,t))=\log f(x,t)\nabla D(x)+D(x)\nabla\log f(x,t)$,
 hence formula \eqref{eq:2-0-13} becomes \eqref{eq:2-0-5}.
 Next, applying integration by parts with the natural boundary condition
 \eqref{eq:2-0-4}, we obtain,
 \begin{multline}
  \label{eq:2-0-7}
  \int_\Omega
  \left(
  D(x)\log f(x,t)
  +
  \phi(x)
  \right)
  \nabla\cdot(f(x,t)\vec{u})
  \,dx  \\
  =
  -
  \int_\Omega
  \nabla
  \left(
  D(x)\log f(x,t)
  +
  \phi(x)
  \right)
  \cdot
  (f(x,t)\vec{u})
  \,dx.
 \end{multline}
 From \eqref{eq:2-0-6}, \eqref{eq:2-0-5}, and \eqref{eq:2-0-7},
 we obtain the energy law,
 \begin{equation*}
  \frac{d}{dt}F[f]
   =
   -
   \int_\Omega
   \frac{(b(x,t))^2}{2D(x)}
   \left|
    \nabla
    \left(
     \phi(x)
     +
     D(x)\log f(x,t)
    \right)
   \right|^2
   f(x,t)
   \,dx.
 \end{equation*}
\end{proof}

One can observe from the energy law \eqref{eq:2-0-8} that an equilibrium state
$f^{\mathrm{eq}}$ for the Fokker-Planck equation \eqref{eq:2-0-4}
satisfies $\nabla(\phi(x)+D(x)\log f^{\mathrm{eq}})=0$. Here,  we derive
the explicit representation of the equilibrium solution for the Fokker-Planck
equation \eqref{eq:2-0-4}.

\begin{proposition}
 Let $b=b(x,t)$, $D=D(x)$, $f_0=f_0(x)$, $\phi=\phi(x)$ be sufficiently smooth
 functions. Then the smooth equilibrium state $f^{\mathrm{eq}}$ for the
 Fokker-Planck equation \eqref{eq:2-0-4} is given by,
 \begin{equation}
  \label{eq:2-0-9}
  f^{\mathrm{eq}}(x)
   =
   \exp
   \left(
   -\frac{\phi(x)-\Cl{const:2.9}}{D(x)}
   \right),
 \end{equation}
 where $\Cr{const:2.9}$ is a constant, which satisfies,
 \begin{equation*}
  \int_\Omega
   \exp
   \left(
    -\frac{\phi(x)-\Cr{const:2.9}}{D(x)}
   \right)
   \,dx
   =
   1.
 \end{equation*}
\end{proposition}

\begin{proof}
 We have from the energy law \eqref{eq:2-0-8} that,
 \begin{equation*}
  0
   =
   \frac{d}{dt}F[f^{\mathrm{eq}}]
   =
   -
   \int_\Omega
   \frac{(b(x,t))^2}{2D(x)}
   \left|
    \nabla
    \left(
     \phi(x)
     +
     D(x)\log f^{\mathrm{eq}}(x)
    \right)
   \right|^2
   f^{\mathrm{eq}}(x)
   \,dx,
 \end{equation*}
 hence $\nabla \left(\phi(x) + D(x)\log f^{\mathrm{eq}}
 \right)=0$. Thus, there is a constant $\Cr{const:2.9}$ such that
 \begin{equation*}
  \phi(x)
   +   
   D(x)
   \log f^{\mathrm{eq}}(x)
   =
   \Cr{const:2.9},
 \end{equation*}
 and hence
 \begin{equation*}
   f^{\mathrm{eq}}(x)
   =
   \exp
   \left(
    -
    \frac{\phi(x)-\Cr{const:2.9}}{D(x)}
   \right).
 \end{equation*}
\end{proof}

\begin{remark}
 Note that the nonlinear Fokker-Planck equation \eqref{eq:2-0-4} can also be derived
 from the dissipation property of the free energy $F[f]$
 \eqref{eq:2-0-12} along with the Fokker-Planck equation,
 \begin{equation}
  \label{eq:2-0-14}
  \frac{\partial f}{\partial t}
   =
   -\nabla\cdot
   \left(
    \vec{a}(x,t)f\right)
   +
   \frac12\nabla\cdot
   \left(
    (b(x,t))^2\nabla f
   \right)
 \end{equation}
 subject to the natural boundary condition,
 $(\vec{a}(x,t)f+\frac12(b(x,t))^2\nabla f)\cdot
    \nu|_{\partial\Omega}=0,$ 
 \cite{Katya-Kamala-Chun-Masashi}. Let us briefly review the
 derivation \cite{Katya-Kamala-Chun-Masashi}. Indeed,
 by \eqref{eq:2-0-14} and using the integration by parts, the rate of
 change of the free energy
 $\frac{d}{dt}F[f]$ is calculated as,
 \begin{equation*}
  \begin{split}
   \frac{d}{dt}F[f]
   &=
   \int_\Omega
   (D(x)\log f(x,t)+\phi(x))
   \frac{\partial f}{\partial t}(x,t)\,dx \\
   &=
   -\int_\Omega
   (D(x)\log f(x,t)+\phi(x))
   \nabla\cdot
   \left(
   \left(
   \vec{a}(x,t)
   -
   \frac12 
   (b(x,t))^2\nabla \log f(x,t)
   \right)
   f(x,t)\right)
   \,dx \\
   &=
   \int_\Omega
   \nabla
   (D(x)\log f(x,t)+\phi(x))
   \cdot
   \left(
   \vec{a}(x,t)
   -
   \frac12 
   (b(x,t))^2\nabla \log f(x,t)
   \right)
   f(x,t)
   \,dx.
  \end{split}
 \end{equation*}
 Since
 \begin{equation*}
  \nabla (D(x)\log f(x,t)+\phi(x))
   =\log f(x,t)\nabla D(x)+D(x)\nabla \log f(x,t)+\nabla\phi(x),
 \end{equation*}
 we obtain the energy dissipation estimate as,
 \begin{equation*}
  \frac{d}{dt}F[f]
   =
   -
   \int_\Omega
   \frac{(b(x,t))^2}{2D(x)}
   \left|
    \nabla (D(x)\log f(x,t)+\phi(x))
   \right|^2
   f(x,t)\,dx
 \end{equation*}
 provided the following relation holds,
 \begin{equation}
  \label{eq:2-0-15}
  \vec{a}(x,t)
   =
   -\frac{(b(x,t))^2}{2D(x)}\nabla\phi(x)
   -\frac{(b(x,t))^2}{2D(x)}\log f(x,t)\nabla D(x).
 \end{equation}
 Note that when $D(x)$ is independent of $x$, $\nabla D(x)=0$ and hence
 \eqref{eq:2-0-4} becomes a \emph{linear} Fokker-Planck
 equation. The relation \eqref{eq:2-0-15} is consistent with the
 \emph{fluctuation-dissipation relation}, which should guarantee not only the
 dissipation property of the free energy $F[f]$,  but also that the
 solution of the nonlinear Fokker-Planck equation \eqref{eq:2-0-4}
 converges to the equilibrium state $f^{\mathrm{eq}}$ given by
 \eqref{eq:2-0-9} (see also \cite{epshteyn2021stochastic} for more
 detailed discussion).
\end{remark}
Now, let us
define the scaled function $\rho$ by taking the ratio of $f$ and
$f^{\mathrm{eq}}$ \eqref{eq:2-0-9},
\begin{equation}
 \label{eq:2-0-11}
 \rho(x,t)
  =
  \frac{f(x,t)}{f^{\mathrm{eq}}(x)},
  \quad
  \text{or}
  \quad
  f(x,t)
  =
  \rho(x,t)
  f^{\mathrm{eq}}(x)
  =
  \rho(x,t)
   \exp
   \left(
    -
    \frac{\phi(x)-\Cr{const:2.9}}{D(x)}
   \right).
\end{equation}
This auxiliary function was also employed in \cite[Theorem
2.1]{MR3497125} to study long-time asymptotics of the solutions of
linear Fokker-Planck equations. Here, we will use the
scaled function $\rho$ as a part of local well-posedness study. Hence,
below, we will reformulate the nonlinear Fokker-Planck equation
\eqref{eq:2-0-4} into a model for the scaled function $\rho$. We have,
\begin{equation*}
 f^{\mathrm{eq}}
 \frac{\partial \rho}{\partial t}
  =
  \nabla\cdot
  \left(
   \frac{(b(x,t))^2}{2D(x)}
   f^{\mathrm{eq}}\rho
    \left(
     \nabla\phi(x)
     +
     \log(f^{\mathrm{eq}}\rho)\nabla D(x)
     +
     D(x)\nabla\log(f^{\mathrm{eq}}\rho)
    \right)
   \right).
\end{equation*}
Next, using the equilibrium state \eqref{eq:2-0-9}, we have,
\begin{equation*}
 \nabla D(x)\log f^{\mathrm{eq}}
  +
  D(x)\nabla (\log f^{\mathrm{eq}})
  +
  \nabla\phi(x)
  =
  0.
\end{equation*}
In addition, note that $\log \rho \nabla D(x)+D(x)\nabla \log \rho =\nabla(D(x)\log
\rho)$. Thus,  the scaled function $\rho$ satisfies,
\begin{equation*}
 f^{\mathrm{eq}}
  \frac{\partial \rho}{\partial t}
  =
  \nabla\cdot
  \left(
   \frac{(b(x,t))^2}{2D(x)}
   f^{\mathrm{eq}}\rho
   \nabla
   \left(
    D(x)\log\rho
   \right)
  \right).
\end{equation*}
Employing the property of the equilibrium state \eqref{eq:2-0-9} again, the natural boundary
condition \eqref{eq:2-0-4} becomes,
\begin{equation*}
  \left(
   \frac{(b(x,t))^2}{2D(x)}
   f^{\mathrm{eq}}\rho
   \nabla
   \left(
    D(x)\log\rho
   \right)
  \right)
  \cdot
  \nu
  \bigg|_{\partial\Omega}
  =
  0.
\end{equation*}
Therefore, the nonlinear Fokker-Planck equation
\eqref{eq:2-0-4} transforms into the following initial-boundary value
problem for $\rho$ defined in \eqref{eq:2-0-11},
\begin{equation}
 \label{eq:2-0-10}
  \left\{
   \begin{aligned}
    &f^{\mathrm{eq}}(x)
    \frac{\partial \rho}{\partial t}
    =
    \nabla\cdot
    \left(
    \frac{(b(x,t))^2}{2D(x)}
    f^{\mathrm{eq}}(x)\rho
    \nabla
    \left(
    D(x)\log\rho
    \right)
    \right),&\quad
    &x\in\Omega,\
    t>0, \\
    &\left(
    \frac{(b(x,t))^2}{2D(x)}
    f^{\mathrm{eq}}(x)\rho
    \nabla
    \left(
    D(x)\log\rho
    \right)
    \right)
    \cdot
    \nu
    \bigg|_{\partial\Omega}
    =
    0,&\quad
    &t>0, \\
    &\rho(0,x)
    =
    \rho_0(x)=\frac{f_0(x)}{f^{\mathrm{eq}}(x)},&
    \quad
    &x\in\Omega.
   \end{aligned}
  \right.
\end{equation}
Next, the free energy \eqref{eq:2-0-12} and the energy law
 \eqref{eq:2-0-8} can also be stated in terms of $\rho$.  Using $D(x)\log f^{\mathrm{eq}}(x)=-\phi(x)+\Cr{const:2.9} $ from
 \eqref{eq:2-0-9}, we obtain,
 \begin{equation}
\label{rhofe}
  F[f]=\int_\Omega
   \left(
    D(x)(\log\rho-1)+\Cr{const:2.9}
   \right)
   \rho f^{\mathrm{eq}}(x)
   \,dx,
 \end{equation}
 and,
 \begin{equation}
\label{rhodfe}
  \frac{d}{dt}F[f]
   =
   -
   \int_\Omega
   \frac{(b(x,t))^2}{2D(x)}
   \left|
    \nabla
    (
    D(x)\log \rho
    )
   \right|^2
   \rho f^{\mathrm{eq}}(x)
   \,dx.
 \end{equation}
 Thus, it is clear from \eqref{rhofe}-\eqref{rhodfe} that weighted
 $L^2$ space, $L^2(\Omega, f^{\mathrm{eq}}(x)\,dx)$ can play
 an important role in studying the equation \eqref{eq:2-0-10} (see for example,
 \cite{MR1812873, epshteyn2021stochastic}).

However, hereafter, we study a classical solution for the problem
\eqref{eq:2-0-10}, and we consider H\"older spaces and norms as
defined below. We give now the notion of a classical solution of
the problem \eqref{eq:2-0-10}.

\begin{definition}
 A function $\rho=\rho(x,t)$ is a classical solution of the problem
 \eqref{eq:2-0-10} in $\Omega\times[0,T)$ if $\rho\in
 C^{2,1}(\Omega\times(0,T))\cap C^{1,0}(\overline{\Omega}\times[0,T))$,
$\rho(x,t)>0$ for $(x,t)\in \Omega\times[0,T)$, and satisfies equation
 \eqref{eq:2-0-10} in a classical sense.
\end{definition}
To state assumptions and the main result, we also define the parabolic
H\"older spaces and norms. For the H\"older exponent
$0<\alpha<1$, the time interval $T>0$, and
the function $f$ on $\Omega\times[0,T)$, we define the
supremum norm $\|f\|_{C(\Omega\times[0,T))}$, the H\"older semi-norms
$[f]_{\alpha,\Omega\times[0,T)}$, and $\langle f\rangle_{\alpha,
\Omega\times[0,T)}$ as,
\begin{equation}
 \begin{split}
  \|f\|_{C(\Omega\times[0,T))}
  &=
  \sup_{x\in\Omega,\ t\in[0,T)}
  |f(x,t)|, \\
  [f]_{\alpha,\Omega\times[0,T)}
  &:=
  \sup_{x,x'\in\Omega,\ t\in[0,T)}
  \frac{|f(x,t)-f(x',t)|}{|x-x'|^\alpha}, \\
  \langle f\rangle_{\alpha,\Omega\times[0,T)}
  &:=
  \sup_{x\in\Omega,\ t,t'\in[0,T)}
  \frac{|f(x,t)-f(x,t')|}{|t-t'|^\alpha},
 \end{split}
\end{equation}
here $|x-x'|$ denotes the euclidean distance between
the vector variables $x$ and $x'$ and $|t-t'|$ denotes
the absolute value of $t-t'$.  For the H\"older
exponent $0<\alpha<1$, the derivative of order $k=0,1,2$, and the time
interval $T>0$, we define the parabolic H\"older spaces
$C^{k+\alpha,(k+\alpha)/2}(\Omega\times[0,T))$ as,
\begin{equation}
 C^{k+\alpha,(k+\alpha)/2}(\Omega\times[0,T))
  :=
  \{f:\Omega\times[0,T)\rightarrow\R,\ \|f\|_{C^{k+\alpha,(k+\alpha)/2}(\Omega\times[0,T))}<\infty\},
\end{equation}
where
\begin{equation}
 \begin{split}
  \|f\|_{C^{\alpha,\alpha/2}(\Omega\times[0,T))}
  &:=
  \|f\|_{C(\Omega\times[0,T))}
  +
  [f]_{\alpha,\Omega\times[0,T)}
  +
  \langle f\rangle_{\alpha/2,\Omega\times[0,T)}, \\
  \|f\|_{C^{1+\alpha,(1+\alpha)/2}(\Omega\times[0,T))}
  &:=
  \|f\|_{C(\Omega\times[0,T))}
  +
  \|\nabla f\|_{C(\Omega\times[0,T))} \\
  &\qquad
  +
  [\nabla f]_{\alpha,\Omega\times[0,T)}
  +
  \langle f\rangle_{(1+\alpha)/2,\Omega\times[0,T)}
  +
  \langle \nabla f\rangle_{\alpha/2,\Omega\times[0,T)}, \\
  \|f\|_{C^{2+\alpha,1+\alpha/2}(\Omega\times[0,T))}
  &:=
  \|f\|_{C(\Omega\times[0,T))}
  +
  \|\nabla f\|_{C(\Omega\times[0,T))}
  +
  \|\nabla^2 f\|_{C(\Omega\times[0,T))}
  +
  \left\|
  \frac{\partial f}{\partial t}
  \right\|_{C(\Omega\times[0,T))} \\
  &\qquad
  +
  [\nabla^2 f]_{\alpha,\Omega\times[0,T)}
  +
  \left[
  \frac{\partial f}{\partial t}
  \right]_{\alpha,\Omega\times[0,T)} \\
  &\qquad
  +
  \langle \nabla f\rangle_{(1+\alpha)/2,\Omega\times[0,T)}
  +
  \langle \nabla^2f\rangle_{\alpha/2,\Omega\times[0,T)}
  +
  \left\langle
  \frac{\partial f}{\partial t}
  \right\rangle_{\alpha/2,\Omega\times[0,T)}.
 \end{split}
\end{equation}
It is well-known that the parabolic H\"older space
$C^{k+\alpha,(k+\alpha)/2}(\Omega\times[0,T))$ is
a Banach space. More properties of the H\"older spaces
can be found in \cite{MR1406091, MR0241822, MR1465184}.  Next, we give
assumptions for the coefficients and the initial data. First, we assume
the strong positivity for the coefficients $b$ and $D$, namely, there
are constants $\Cl{const:2.1}, \Cl{const:2.8}>0$ such that for
$x\in\Omega$ and $t>0$,
\begin{equation}
 \label{eq:2-0-1}
 b(x,t)\geq \Cr{const:2.1},\quad
  D(x)\geq \Cr{const:2.8}.
\end{equation}
Next, we assume the H\"older regularity for $0<\alpha<1$: coefficients
$b(x,t)$, $\phi(x)$, $D(x)$, an initial datum
$\rho_0(x)$ and a domain $\Omega$ satisfy,
\begin{equation}
 \label{eq:2-0-2}
  b^2\in C^{1+\alpha,(1+\alpha)/2}(\Omega\times[0,T)),\
  \phi\in C^{2+\alpha}(\Omega),\
  D\in C^{2+\alpha}(\Omega),\
  \partial\Omega\ \text{is}\ C^{2+\alpha},\
  \text{and}\
  \rho_0\in C^{2+\alpha}(\Omega).
\end{equation}
As a consequence of the above assumptions, $f^{\mathrm{eq}}$ is in
$C^{2+\alpha}(\Omega)$. Finally, assume the compatibility condition
for the initial data
$\rho_0$,
\begin{equation}
 \label{eq:2-0-3}
  \nabla( D(x)\log\rho_0) \cdot \nu
  \bigg|_{\partial\Omega}
  =
  0.
\end{equation}
Since $b(x,t)$, $D(x)$, $f^{\mathrm{eq}}$, and $\rho$ are
positive, \eqref{eq:2-0-3} is sufficient for the compatibility condition of
\eqref{eq:2-0-10}.

Now we are ready to state the main theorem about existence of a classical solution
of \eqref{eq:2-0-10}.

\begin{theorem}
 \label{thm:2-0-1} 
 Let coefficients $b(x,t)$, $\phi(x)$, $D(x)$, a positive probability
 density function $\rho_0(x)$ and a bounded domain $\Omega$ satisfy the
 strong positivity \eqref{eq:2-0-1}, the H\"older regularity
 \eqref{eq:2-0-2} for $0<\alpha<1$, and the compatibility for the
 initial data \eqref{eq:2-0-3}, respectively. Then,  there exist a time interval $T>0$
 and a classical solution $\rho=\rho(x,t)$ of \eqref{eq:2-0-10} on
 $\Omega\times[0,T)$ with the H\"older regularity $\rho\in
 C^{2+\alpha,1+\alpha/2}(\Omega\times[0,T))$.
\end{theorem}
\begin{corollary}
\label{cor:main_th}
  Let coefficients $b(x,t)$, $\phi(x)$, $D(x)$, and a bounded domain
 $\Omega$ satisfy the strong positivity \eqref{eq:2-0-1} and the H\"older
 regularity \eqref{eq:2-0-2} for $0<\alpha<1$, respectively. Let $f_0$ be a positive
 probability density function from $C^{2+\alpha}(\Omega)$, which
 is positive everywhere, and satisfies the compatibility condition,
 \begin{equation*}
  \nabla
   \left(
    \phi(x)
    +
    \log (D(x)f_0)
   \right)
   \cdot
  \nu
  \bigg|_{\partial\Omega}
  =
  0.
 \end{equation*}
 Then,  there exist a time interval $T>0$ and a classical solution
 $f=f(x,t)$ of \eqref{eq:2-0-4} on $\Omega\times[0,T)$ with the H\"older
 regularity $f\in C^{2+\alpha,1+\alpha/2}(\Omega\times[0,T))$.
\end{corollary}
Before we proceed with a proof of the Theorem \ref{thm:2-0-1}, and
hence Corollary \ref{cor:main_th}, we give
a brief overview of the main ideas of the proof:
\begin{enumerate}[1.]
 \item In Section \ref{sec:2-1}, we consider the change of variables $h$
       in \eqref{eq:2-1-3} and $\xi$ in \eqref{eq:2-1-2}. We will derive
       evolution equations in terms of $h$ and $\xi$ in Lemma
       \ref{lem:2-1-2} and Lemma \ref{lem:2-1-1}. Note that,  $\xi$ vanishes
       at $t=0$, namely, we have, $\xi(x,0)=0$.
 \item In Section \ref{sec:2-2}, we give the decay properties of the
       H\"older norms  $\|\nabla\xi\|_{C^{\alpha.\alpha/2}(\Omega)\times[0,T)}$ and
       $\|\xi\|_{C^{\alpha.\alpha/2}(\Omega)\times[0,T)}$ in terms of
       $\xi$, see \eqref{eq:2-2-2} and \eqref{eq:2-2-3}. Thanks to the condition that $\xi(x,0)=0$, we can obtain
       explicit decay of $\|\nabla
       \xi\|_{C^{\alpha.\alpha/2}(\Omega)\times[0,T)}$ and
       $\|\xi\|_{C^{\alpha.\alpha/2}(\Omega)\times[0,T)}$.
 \item In Section \ref{sec:2-3}, we study a linear parabolic equation
       \eqref{eq:2-3-8} associated with the nonlinear problem
       \eqref{eq:2-1-4}. We show that for the appropriate choice of constants $M,T>0$
       and for $\psi\in X_{M,T}$, where $X_{M,T}$ is defined in
       \eqref{eq:2-3-9}, a solution $\xi$ of \eqref{eq:2-3-8} 
       belongs to $X_{M,T}$, see Lemma \ref{lem:2-3-1}. Thus,  we can
       define a solution map $A:\psi\mapsto\xi$ on $X_{M,T}$.
 \item In Section \ref{sec:2-4}, we show that the solution map has the
       contraction property, see Lemma \ref{lem:2-4-1}. In order to show
       that the Lipschitz constant is less than $1$, we use the decay
       properties of the H\"older norms \eqref{eq:2-2-2},
       \eqref{eq:2-2-3}.
\item Since the solution map is a contraction mapping on $X_{M,T}$, there is a
fixed point $\xi\in X_{M,T}$. The fixed point is a classical solution of
\eqref{eq:2-1-4},  hence we can find a classical solution of
\eqref{eq:2-0-10}. Once we find a solution $\rho$ of \eqref{eq:2-0-10}, by the definition of
the scaled function \eqref{eq:2-0-11}, we obtain a solution of
\eqref{eq:2-0-4}. Note, that in Section \ref{sec:3}, we show
uniqueness of a local solution of the problem \eqref{eq:2-0-10}, and
hence of a local solution of the problem \eqref{eq:2-0-4}.
\end{enumerate}
\subsection{Change of variables}
\label{sec:2-1}
The problem \eqref{eq:2-0-10} is well defined only when $\rho>0$.
However,  it is difficult to prove the positivity of $\rho$ using
\eqref{eq:2-0-10} directly due to lack of maximum principle for 
the nonlinear models. Instead, we will construct a solution $\rho$ of
\eqref{eq:2-0-10}, and will guarantee the positivity of $\rho$, by
introducing a new auxiliary variable $h$ as follows,
\begin{equation}
 \label{eq:2-1-3}
  h(x,t)=D(x)\log\rho(x,t),\quad
  \text{or}\quad
  \rho(x,t)
  =
  \exp
  \left(
   \frac{h(x,t)}{D(x)}
  \right).
\end{equation}
Once we find a solution $h$, then we can obtain a solution
$\rho$ of \eqref{eq:2-0-10} using the change of variables as in
\eqref{eq:2-1-3}. Furthermore, we will
show uniqueness of a local solution $\rho$ in Section \ref{sec:3}.

\par Let us derive the evolution equation in terms of the new variable $h$ in
\eqref{eq:2-1-3}.


\begin{lemma}
 \label{lem:2-1-2} 
 Let $\rho$ be a classical solution of \eqref{eq:2-0-10} and define $h$
 as in \eqref{eq:2-1-3}. Then,  the auxiliary variable $h$ satisfies the
 following equation in a classical sense,
\begin{equation}
  \label{eq:2-1-5}
   \left\{
    \begin{aligned}
     \frac{f^{\mathrm{eq}}(x)}{D(x)}
     \frac{\partial h}{\partial t}
     &=
    \nabla\cdot
     \left(
     \frac{(b(x,t))^2}{2D(x)}
     f^{\mathrm{eq}}(x)
     \nabla h
     \right)
     +
     \frac{(b(x,t))^2}{2D(x)}
     f^{\mathrm{eq}}(x)
     \nabla h
     \cdot
     \nabla
     \left(
     \frac{h}{D(x)}
     \right)
     ,\quad
     x\in\Omega,\
     t>0, \\
     \nabla
     h
     \cdot
     \nu
     \bigg|_{\partial\Omega}
     &=
     0,\quad
     t>0, \\
     h(0,x)
     &=
     h_0(x)
     =
     D(x)\log \rho_0(x),
     \quad
     x\in\Omega.
    \end{aligned}
   \right.
 \end{equation}

 Conversely, let $h\in C^{2,1}(\Omega\times(0,T))\cap
 C^{1,0}(\overline{\Omega}\times[0,T))$ be a solution of
 \eqref{eq:2-1-5} in a classical sense and define $\rho$ as
 \eqref{eq:2-1-3}. Then, $\rho$ is a classical solution of
 \eqref{eq:2-0-10}.
\end{lemma}

\begin{proof}
 By straightforward calculation of the derivative of $\rho$ using
 \eqref{eq:2-1-3}, we have that $\rho_t=\frac{e^{h/D(x)}}{D(x)} h_t$,
 as well as,
 \begin{equation*}
  \frac{(b(x,t))^2}{2D(x)}
   f^{\mathrm{eq}}(x)\rho
   \nabla
   \left(
   D(x)\log\rho
   \right)
   =
   \frac{(b(x,t))^2}{2D(x)}
   f^{\mathrm{eq}}(x)
   e^{h/D(x)}
   \nabla
   h,
 \end{equation*}
and,
 \begin{equation*}
  \begin{split}
   &\quad
   \nabla\cdot
   \left(
   \frac{(b(x,t))^2}{2D(x)}
   f^{\mathrm{eq}}(x)\rho
   \nabla
   \left(
   D(x)\log\rho
   \right)
   \right) \\
   &=
   e^{h/D(x)}
   \nabla\cdot
   \left(
   \frac{(b(x,t))^2}{2D(x)}
   f^{\mathrm{eq}}(x)
   \nabla
   h
   \right)
   +
   \frac{(b(x,t))^2}{2D(x)}
   f^{\mathrm{eq}}(x)
   e^{h/D(x)}
   \nabla
   h
   \cdot
   \nabla
   \left(
   \frac{h}{D(x)}
   \right).
  \end{split}
 \end{equation*}
 Note that $b$, $D$, $f^{\mathrm{eq}}$, and $e^{h/D}$ are positive
 functions, hence the boundary condition of the model \eqref{eq:2-0-10} is
 equivalent to the Neumann boundary condition for the function $h$. Using these
 relations, we obtain result of Lemma \ref{lem:2-1-2}.
\end{proof}

\begin{remark}
Note, employing  the change of the variable for $\rho$ in terms of $h$ \eqref{eq:2-1-3}, the free
energy $F[f]$ \eqref{rhofe} and
 the dissipation law \eqref{rhodfe} are transformed into,
 \begin{equation}
\label{hfe}
  F[f]
  =
  \int_{\Omega}
  \left(
   h(x,t)-D(x)+\Cr{const:2.9}
  \right)
  \exp
  \left(
   \frac{h(x,t)}{D(x)}
  \right)
  f^{\mathrm{eq}}(x)   
  \,dx,
 \end{equation}
 and,
 \begin{equation}
\label{hdfe}
  \frac{d}{dt}F[f]
   =
   -
   \int_\Omega
   \frac{(b(x,t))^2}{2D(x)}
   \left|
    \nabla
    h(x,t)
   \right|^2
   \exp
   \left(
    \frac{h(x,t)}{D(x)}
   \right)
   f^{\mathrm{eq}}(x)   
   \,dx.
 \end{equation}
\end{remark}
\begin{remark}
  The non-linearity of the problem \eqref{eq:2-1-5} is the so-called \emph{scale critical}. The
 diffusion term $\Delta h$ and the nonlinear term $|\nabla h|^2$ have
 the same scale. To see this, for $\gamma>0$ we consider the following
 equation,
 \begin{equation}
  \label{eq:2-1-7}
 \frac{\partial u}{\partial t}(x,t)
   =
   \Delta u(x,t)
   +
   |\nabla u(x,t)|^\gamma,
   \quad
   x\in\R^d,\quad
   t>0.
 \end{equation}
 For a positive scaling parameter $\lambda>0$ and
 $(x_0,t_0)\in\R^d\times(0,\infty)$, let us consider the change of
 variables $x-x_0=\lambda y$, $t-t_0=\lambda^2 s$, and a scale transformation
 $v(y,s)=u(x,t)$. Then,
 \begin{equation*}
  \frac{\partial u}{\partial t}(x,t)=
   \frac{1}{\lambda^2}\frac{\partial v}{\partial s}(y,s),\quad
   \Delta_x u(x,t)=\frac{1}{\lambda^2}\Delta_y v(y,s),\quad
   |\nabla_x u(x,t)|^\gamma=\frac{1}{\lambda^\gamma} |\nabla_y v(y,s)|^\gamma,
 \end{equation*}
 hence the scale transformation $v$ satisfies,
 \begin{equation*}
  \frac{\partial v}{\partial s}(y,s)
   =
   \Delta_y v(y,s)
   +
   \lambda^{2-\gamma}
   |\nabla v(y,s)|^\gamma,
   \quad
   y\in\R^d,\quad
   0<s<t_0.
 \end{equation*}
 When we take $\lambda\downarrow0$, the function $u(x,t)$ will blow-up
 at $x=x_0$, and is regarded as a perturbation of a linear function
 around $x=x_0$. If $\gamma<2$, which is called \emph{scale
 sub-critical}, then $\lambda^{2-\gamma}\rightarrow0$ as
 $\gamma\downarrow0$.  Hence, the non-linearity $|\nabla u(x,t)|^\gamma$
 can be regarded as a small perturbation in terms of the diffusion term
 $\Delta u(x,t)$. If $\gamma>2$, which is called \emph{scale
 super-critical}, then $\lambda^{\gamma-2}\rightarrow0$ as
 $\gamma\downarrow0$. In this case, the non-linear term $|\nabla
 u(x,t)|^\gamma$ becomes a principal term. Thus the behavior of $u$ may
 be different from solutions of the linear problem, namely, the
 solutions of the heat equation. If $\gamma=2$, which is called
 \emph{scale critical case}, then $\lambda^{2-\gamma}=1$ (like in our
 model \eqref{eq:2-1-5}). The diffusion term $\Delta u(x,t)$ and the
 nonlinear term $|\nabla u(x,t)|^2$ are balanced, hence the
 non-linearity $|\nabla u(x,t)|^2$ cannot be regarded as the small
 perturbation anymore, especially for the study of the global existence
 and long-time asymptotic behavior. Thus, in the problem
 \eqref{eq:2-1-5}, we need to consider the interaction between the
 diffusion term and the nonlinear term accurately. For the importance of
 the scale transformation, see for instance \cite{MR2656972,
   MR0784476}. The scale critical case for \eqref{eq:2-1-7} is related to the heat flow for harmonic maps. See for instance,
 \cite{MR1324408,MR1266472,MR0990191,MR2155901}. See also 
 \cite{MR0164306,MR0664498} for the steady-state case.
\end{remark}

 Our goal is to use the Schauder estimates for linear parabolic
equations, therefore
we rewrite \eqref{eq:2-1-5} in the non-divergence form,
\begin{equation*}
   \frac{\partial h}{\partial t}
   =
   \frac{(b(x,t))^2}{2}
   \Delta h
   +
   \frac{D(x)}{f^{\mathrm{eq}}(x)}
   \nabla
   \left(
   \frac{(b(x,t))^2}{2D(x)}
   f^{\mathrm{eq}}(x)
   \right)
   \cdot
   \nabla h
   +
   \frac{(b(x,t))^2}{2D(x)}
   |\nabla h|^2
   -
   \frac{(b(x,t))^2}{2(D(x))^2}
   h
   \nabla h
   \cdot
   \nabla
   D(x).
\end{equation*}

Next, we introduce a new variable $\xi$ as,
\begin{equation}
 \label{eq:2-1-2}
 h(x,t)=h_0(x)+\xi(x,t),
\end{equation}
in order to change problem \eqref{eq:2-1-5} into the zero initial value
problem with $\xi(x,0)=0$. Note that, when $h$ is sufficiently close to the initial
data $h_0$ for small $t>0$ in the H\"older space, $\xi$ should be also
small enough for small $t>0$. To show the smallness of the
nonlinearity in the H\"older space, we consider the nonlinear terms in
terms of $\xi$ instead of $h$. Thus, below, we will derive the evolution equation in
terms of $\xi$.

\begin{lemma}
 \label{lem:2-1-1}
 %
 Let $h\in C^{2,1}(\Omega\times(0,T))\cap
 C^{1,0}(\overline{\Omega}\times[0,T))$ be a solution of
 \eqref{eq:2-1-5} in a classical sense and define $\xi$ as
 in \eqref{eq:2-1-2}. Then, $\xi$ satisfies the following equation in a
 classical sense,
 \begin{equation}
  \label{eq:2-1-4}
  \left\{
   \begin{aligned}
    \frac{\partial \xi}{\partial t}
    &=
    L\xi+g_0(x,t)+G(\xi),\quad
    x\in\Omega,\
    t>0, \\
    \nabla
    \xi
    \cdot
    \nu
    \bigg|_{\partial\Omega}
    &=
    0,\quad
    t>0, \\
    \xi(0,x)
    &=
    0,
    \quad
    x\in\Omega,
   \end{aligned}
  \right.
 \end{equation}
 where
 \begin{equation}
  \label{eq:2-1-1}
   \begin{split}
    L\xi
    &:=
    \frac{(b(x,t))^2}{2}
    \Delta \xi \\
    &\quad
    +
    \left(
    \frac{D(x)}{f^{\mathrm{eq}}(x)}
    \nabla
    \left(
    \frac{(b(x,t))^2}{2D(x)}
    f^{\mathrm{eq}}(x)
    \right)
    +
    \frac{(b(x,t))^2}{D(x)}\nabla h_0(x)
    -
    \frac{(b(x,t))^2h_0(x)}{2(D(x))^2}\nabla D(x)
    \right)
    \cdot
    \nabla \xi
    \\
    &\quad
    -
    \left(
    \frac{(b(x,t))^2}{2(D(x))^2}\nabla D(x)\cdot\nabla h_0(x)
    \right)
    \xi, \\
    g_0(x,t)
    &:=
    \frac{(b(x,t))^2}{2}
    \Delta h_0(x)
    +
    \frac{D(x)}{f^{\mathrm{eq}}(x)}
    \nabla
    \left(
    \frac{(b(x,t))^2}{2D(x)}
    f^{\mathrm{eq}}(x)
    \right)
    \cdot
    \nabla h_0(x) \\
    &\quad
    +
    \frac{(b(x,t))^2}{2D(x)}
    |\nabla h_0(x)|^2
    -
    \frac{(b(x,t))^2}{2(D(x))^2}
    h_0(x)
    \nabla h_0(x)
    \cdot
    \nabla
    D(x), \\
    G(\xi)
    &:=
    \frac{(b(x,t))^2}{2D(x)}
    f^{\mathrm{eq}}(x)
    |\nabla \xi|^2
    -
    \frac{(b(x,t))^2}{2(D(x))^2}
    f^{\mathrm{eq}}(x)
    \xi
    \nabla \xi
    \cdot
    \nabla
    D(x).
   \end{split}
 \end{equation}
 
 Conversely, let $\xi\in C^{2,1}(\Omega\times(0,T))\cap
 C^{1,0}(\overline{\Omega}\times[0,T))$ be a solution of
 \eqref{eq:2-1-4} in a classical sense and define $h$ as in
 \eqref{eq:2-1-2}. Then, $h$ is a solution of \eqref{eq:2-1-5} in a
 classical sense.
\end{lemma}

\begin{proof}
The equivalence of the initial conditions for functions $h$ and $\xi$ is trivial, so we consider the
 equivalence of the differential equations and of the boundary conditions
 for $h$ and $\xi$.
 First, we derive the differential equation for $\xi$ using the change
 of variable in \eqref{eq:2-1-2}.  Assume $h$ is a
 solution of \eqref{eq:2-1-5} in a classical sense. Since $\xi_t=h_t$,
 $\nabla h=\nabla h_0+\nabla\xi$, $\Delta h=\Delta h_0+\Delta\xi$, we
 have,
 \begin{equation}
  \label{eq:2-1-6}
   \begin{split}
    \frac{\partial \xi}{\partial t}
    &=
    \frac{(b(x,t))^2}{2}
    \Delta \xi
    +
    \frac{D(x)}{f^{\mathrm{eq}}(x)}
    \nabla
    \left(
    \frac{(b(x,t))^2}{2D(x)}
    f^{\mathrm{eq}}(x)
    \right)
    \cdot
    \nabla \xi
    \\
    &\quad
    +\frac{(b(x,t))^2}{2}
    \Delta h_0(x)
    +
    \frac{D(x)}{f^{\mathrm{eq}}(x)}
    \nabla
    \left(
    \frac{(b(x,t))^2}{2D(x)}
    f^{\mathrm{eq}}(x)
    \right)
    \cdot
    \nabla h_0(x)
    \\
    &\quad
    +
    \frac{(b(x,t))^2}{2D(x)}
    |\nabla \xi+\nabla h_0(x)|^2
    -
    \frac{(b(x,t))^2}{2(D(x))^2}
    (\xi+h_0(x))
    \nabla (\xi+h_0(x))
    \cdot
    \nabla
    D(x).
   \end{split}
 \end{equation}
 Using the following relations,
 \begin{equation*}
  \begin{split}
   |\nabla \xi+h_0(x)|^2
   &=
   |\nabla \xi|^2
   +
   2\nabla h_0(x)\cdot\nabla \xi
   +
   |\nabla h_0(x)|^2, \\
   (\xi+h_0(x))
   \nabla (\xi+h_0(x))
   &=
   \xi\nabla\xi
   +
   \xi\nabla h_0(x)
   +
   h_0(x) \nabla \xi
   +
   h_0(x)\nabla h_0(x),
  \end{split}
 \end{equation*}
 the equation \eqref{eq:2-1-6} is transformed into the equation,
 \begin{equation*}
   \begin{split}
    \frac{\partial \xi}{\partial t}
    &=
    \frac{(b(x,t))^2}{2}
    \Delta \xi
    +
    \left(
    \frac{D(x)}{f^{\mathrm{eq}}(x)}
    \nabla
    \left(
    \frac{(b(x,t))^2}{2D(x)}
    f^{\mathrm{eq}}(x)
    \right)
    +
    \frac{(b(x,t))^2}{D(x)}\nabla h_0(x)
    -
    \frac{(b(x,t))^2h_0(x)}{2(D(x))^2}\nabla D(x)
    \right)
    \cdot
    \nabla \xi
    \\
    &\quad
    -
    \left(
    \frac{(b(x,t))^2}{2(D(x))^2}\nabla D(x)\cdot\nabla h_0(x)
    \right)
    \xi \\
    &\quad
    +\frac{(b(x,t))^2}{2}
    \Delta h_0(x)
    +
    \frac{D(x)}{f^{\mathrm{eq}}(x)}
    \nabla
    \left(
    \frac{(b(x,t))^2}{2D(x)}
    f^{\mathrm{eq}}(x)
    \right)
    \cdot
    \nabla h_0(x) \\
    &\quad
    +
    \frac{(b(x,t))^2}{2D(x)}
    |\nabla h_0(x)|^2
    -
    \frac{(b(x,t))^2}{2(D(x))^2}
    h_0(x)
    \nabla h_0(x)
    \cdot
    \nabla
    D(x) \\
    &\quad
    +
    \frac{(b(x,t))^2}{2D(x)}
    f^{\mathrm{eq}}(x)
    |\nabla \xi|^2
    -
    \frac{(b(x,t))^2}{2(D(x))^2}
    f^{\mathrm{eq}}(x)
    \xi
    \nabla \xi
    \cdot
    \nabla
    D(x) \\
    &=
    L\xi+g_0(x,t)+G(\xi).
   \end{split}
 \end{equation*}
 Thus, we obtain the equivalence of the differential equations for $h$
 and $\xi$.

 Next, we consider boundary condition
 $\nabla\xi\cdot\nu|_{\partial\Omega}=0$. Using the compatibility
 condition \eqref{eq:2-0-3}, we have,
 \begin{equation*}
  \nabla \xi\cdot\nu \bigg|_{\partial\Omega}
   =
   \nabla h\cdot\nu \bigg|_{\partial\Omega}
   -
   \nabla h_0\cdot\nu \bigg|_{\partial\Omega}
   =
   \nabla h\cdot\nu \bigg|_{\partial\Omega},
 \end{equation*}
 hence we also have the equivalence of the boundary conditions for $h$
 and $\xi$.
\end{proof}

\begin{remark}
 %
 From the change of variable \eqref{eq:2-1-2}, the free energy $F[f]$ \eqref{hfe}
 and the energy dissipation law \eqref{hdfe} are given in terms of $\xi$ below,
 \begin{equation}
  F[f]
  =
  \int_{\Omega}
  \left(
   \xi(x,t)+h_0(x)-D(x)+\Cr{const:2.9}
  \right)
  \exp
  \left(
   \frac{\xi(x,t)+h_0(x)}{D(x)}
  \right)
  f^{\mathrm{eq}}(x)   
  \,dx,
 \end{equation}
 and
 \begin{equation}
  \frac{d}{dt}F[f]
   =
   -
   \int_\Omega
   \frac{(b(x,t))^2}{2D(x)}
   \left|
    \nabla
    \xi(x,t)
    +
    \nabla h_0(x)
   \right|^2
   \exp
   \left(
    \frac{\xi(x,t)+h_0(x)}{D(x)}
   \right)
   f^{\mathrm{eq}}(x)   
   \,dx.
 \end{equation}
\end{remark}

\begin{remark}
 %
 The idea to consider the variable $\xi$ in \eqref{eq:2-1-2}, in order
 to change \eqref{eq:2-1-5} into the zero initial value problem
 \eqref{eq:2-1-4}, is similar to the study of the inhomogeneous Dirichlet
 boundary value problems for the elliptic equations, see \cite[Theorem 6.8, Theorem
 8.3]{MR1814364}.
\end{remark}

In this section, we made several changes of variables. Hereafter we study
\eqref{eq:2-1-4} with the homogeneous Neumann boundary condition and
with the zero
initial condition. As one can observe in \eqref{eq:2-1-1}, the initial
data $h_0$ (or equivalently $\rho_0$) is included into the coefficients of the linear operator $L$
and of 
the external force $g_0$ of the problem \eqref{eq:2-1-4}.

\subsection{Properties of the H\"older spaces with the zero initial condition}
\label{sec:2-2}

In this section, we study properties of the H\"older spaces with the
zero initial value condition. The main idea behind the proof of the Theorem
\ref{thm:2-0-1} is to find a solution of the problem \eqref{eq:2-1-4} in a function
space as defined below,
\begin{equation}
 \label{eq:2-3-9}
 X_{M,T}
  :=
  \left\{
   \zeta\in C^{2+\alpha,1+\alpha/2}(\Omega\times[0,T))
   :
   \zeta(x,0)=0\ \text{for}\ x\in\Omega,\ 
   \nabla\cdot\zeta\big|_{\partial\Omega}=0,\ 
   \|\zeta\|_{C^{2+\alpha,1+\alpha/2}(\Omega\times[0,T))}\leq M
   \right\}
\end{equation}
for the appropriate  choice of constants $M, T>0$.  

For $\psi\in X_{M,T}$, let $\eta$ be a classical solution of the
following linear parabolic problem,
\begin{equation}
 \label{eq:2-3-8}
 \left\{
  \begin{aligned}
   \frac{\partial \eta}{\partial t}
   &=
   L\eta+g_0(x,t)+G(\psi),\quad
   x\in\Omega,\
   t>0, \\
   \nabla
   \eta
   \cdot
   \nu
   \bigg|_{\partial\Omega}
   &=
   0,\quad
   t>0, \\
   \eta(0,x)
   &=
   0,
   \quad
   x\in\Omega,
  \end{aligned}
 \right.
\end{equation}
where $L$, $g_0(x,t)$ and $G$ are defined in \eqref{eq:2-1-1}. Note
that,  in
Section~\ref{sec:2-3} our goal will be to select constants $M,T>0$ such that for any $\psi\in X_{M,T}$, a
solution $\eta$ belongs to $X_{M,T}$. Thus, here we first need to
introduce the idea of the solution map
and the well-definedness of the solution map on $X_{M,T}$.

\begin{definition}\label{def:solmap}
 For $\psi\in X_{M,T}$, let $\eta=A\psi$ be a solution of
 \eqref{eq:2-3-8}. We call $A$ a solution map for
 \eqref{eq:2-3-8}. The solution map $A$ is\emph{ well-defined} on $X_{M,T}$ if
 $A\psi\in X_{M,T}$ for all $\psi\in X_{M,T}$.
\end{definition}

Once we will show that the solution map $A$ is well-defined in $X_{M,T}$ and
is a contraction for the appropriate choices of constants,  then we can
find a fixed point $\xi\in X_{M,T}$ for the solution map $A$, and thus
establish that $\xi$ is a classical solution of the problem
\eqref{eq:2-1-4}. In order to derive the contraction property of the
solution map $A$, first, we obtain the decay estimates for the H\"older's
norm for $\zeta\in X_{M,T}$.

As we noted in the Remark \ref{rem:2-1-1} below, when a function $\theta\in
C^{\alpha,\alpha/2}(\Omega\times[0,T))$ satisfies $\theta(x,0)=0$, the
supremum norm of $\theta$ and its derivatives will vanish at $t=0$, namely
\begin{equation*}
  \sup_{\Omega\times[0,T)}|\theta|,\
   \sup_{\Omega\times[0,T)}|\nabla \theta|,\
   \sup_{\Omega\times[0,T)}|\nabla^2 \theta|
   \rightarrow0,\quad \text{as}\ T\rightarrow0,
\end{equation*}
as a consequence of the H\"older's norm's estimates \eqref{eq:2-2-2} and
\eqref{eq:2-2-3} obtained below.  Note again that $\theta(x,0)=0$ is
essential for the above convergence. In order to consider the nonlinear
model \eqref{eq:2-1-4} as a perturbation of the linear system \eqref{eq:2-3-8}, we need some smallness
for the norm in general. Hence, we next show explicit decay estimates
for the H\"older's norms which can be applied for a function $\zeta\in X_{M,T}$.

\begin{lemma}
 \label{lem:2-2-1}
 Let any function $\theta\in C^{2+\alpha,1+\alpha/2}(\Omega\times[0,T))$,
 $\theta(x,0)=0$ for $x\in\Omega$.
 Then,
 \begin{equation}
  \label{eq:2-2-2}
  \|\nabla\theta\|_{C^{\alpha,\alpha/2}(\Omega\times[0,T))}
   \leq
   3 (T^{(1+\alpha)/2}+T^{1/2})\|\theta\|_{C^{2+\alpha,1+\alpha/2}(\Omega\times[0,T))}.
 \end{equation}
 Thus, for a function $\zeta\in X_{M,T}$,~\eqref{eq:2-2-2} also holds.
\end{lemma}

\begin{proof}
 First, we consider $\|\nabla\theta\|_{C(\Omega\times[0,T))}$. For
 $x\in\Omega$ and $t\in(0,T)$, we have, by $\nabla\theta(x,0)=0$ and the
 definition of H\"older's norm, that,
 \begin{equation}
  \label{eq:2-2-4}
   |\nabla\theta(x,t)|
   =
   \frac{|\nabla\theta(x,t)-\nabla\theta(x,0)|}{|t-0|^{(1+\alpha)/2}}
   |t-0|^{(1+\alpha)/2}
   \leq
    t^{(1+\alpha)/2}\langle\nabla \theta\rangle_{(1+\alpha)/2,\Omega\times[0,T)}.
 \end{equation}
 Therefore, we have,
 \begin{equation}
  \label{eq:2-2-5}
  \|\nabla\theta\|_{C(\Omega\times[0,T))}
   \leq
   T^{(1+\alpha)/2}\|\theta\|_{C^{2+\alpha,1+\alpha/2}(\Omega\times[0,T))}.
 \end{equation}

 Next, we derive the estimate of
 $[\nabla\theta]_{\alpha,\Omega\times[0,T)}$. For $x,x'\in\Omega$ and
 $t\in(0,T)$, we first assume that $|x-x'|<t^{1/2}$.  Then,
 since we assume that $\Omega$ is convex, the
 fundamental theorem of calculus and the triangle inequality lead to,
 \begin{equation*}
  \begin{split}
   |\nabla\theta(x,t)-\nabla\theta(x',t)|
   &=
   \left|
   \int_0^1
   \frac{d}{d\tau}
   \nabla\theta(\tau x+(1-\tau)x',t)
   \,d\tau
   \right| \\
   &\leq
    |x-x'|\int_0^1
   |\nabla^2\theta(\tau x+(1-\tau)x',t)|
   \,d\tau. \\
  \end{split}
 \end{equation*}
 Since $\nabla^2\theta(\tau x+(1-\tau)x',0)=0$, we have,
 \begin{equation*}
  \begin{split}
   |\nabla^2\theta(\tau x+(1-\tau)x',t)|
   &\leq
   \frac{|\nabla^2\theta(\tau x+(1-\tau)x',t)-\nabla^2\theta(\tau x+(1-\tau)x',0)|}{|t-0|^{\alpha/2}}|t-0|^{\alpha/2} \\
   &\leq
   T^{\alpha/2}\langle
   \nabla^2\theta
   \rangle_{\alpha/2,\Omega\times[0,T)}.
  \end{split}
 \end{equation*}
 Using the assumption $|x-x'|<t^{1/2}$,  and that $|x-x'|=|x-x'|^{1-\alpha}|x-x'|^{\alpha}$, we conclude,

 \begin{equation}
  \label{eq:2-2-6}
  |\nabla\theta(x,t)-\nabla\theta(x',t)|
   \leq
   T^{\alpha/2}
   t^{(1-\alpha)/2}\langle
   \nabla^2\theta
   \rangle_{\alpha/2,\Omega\times[0,T)}
   |x-x'|^\alpha
   \leq
   T^{1/2}\langle
   \nabla^2\theta
   \rangle_{\alpha/2,\Omega\times[0,T)}
   |x-x'|^\alpha.
 \end{equation}
 Next, we consider the case $|x-x'|\geq t^{1/2}$. Using \eqref{eq:2-2-4},
 we have,
 \begin{equation*}
   |\nabla\theta(x,t)|
   =
   \frac{|\nabla\theta(x,t)-\nabla\theta(x,0)|}{|t-0|^{(1+\alpha)/2}}
   |t-0|^{(1+\alpha)/2}
   \leq
   T^{1/2}\langle\nabla \theta\rangle_{(1+\alpha)/2,\Omega\times[0,T)}
   |x-x'|^\alpha,
 \end{equation*}
 hence we obtain,
 \begin{equation}
  \label{eq:2-2-7}
  |\nabla\theta(x,t)-\nabla\theta(x',t)|
   \leq
   |\nabla\theta(x,t)|+|\nabla\theta(x',t)|
   \leq
   2  T^{1/2}\langle\nabla \theta\rangle_{(1+\alpha)/2,\Omega\times[0,T)}|x-x'|^\alpha.
 \end{equation}
 Combining \eqref{eq:2-2-6} and \eqref{eq:2-2-7} we arrive at,
 \begin{equation}
  \label{eq:2-2-8}
  [\nabla\theta]_{\alpha,\Omega\times[0,T)}
   \leq
   2  T^{1/2}\|\theta\|_{C^{2+\alpha,1+\alpha/2}(\Omega\times[0,T))}.
 \end{equation}

 Finally, we consider
 $\langle\nabla\theta\rangle_{\alpha/2,\Omega\times[0,T)}$. For
 $x\in\Omega$ and $t,t'\in(0,T)$, we have
 \begin{equation*}
  |\nabla\theta(x,t)-\nabla\theta(x,t')|
   \leq
   \frac{|\nabla\theta(x,t)-\nabla\theta(x,t')|}{|t-t'|^{(1+\alpha)/2}}|t-t'|^{(1+\alpha)/2}
   \leq
   T^{1/2}|t-t'|^{\alpha/2}\langle\nabla\theta\rangle_{(1+\alpha)/2,\Omega\times[0,T)},
 \end{equation*}
 hence
 \begin{equation}
  \label{eq:2-2-9}
  \langle\nabla\theta\rangle_{\alpha/2,\Omega\times[0,T)}
   \leq
   T^{1/2}\|\theta\|_{C^{2+\alpha,1+\alpha/2}(\Omega\times[0,T))}.
 \end{equation}
 Combining \eqref{eq:2-2-5}, \eqref{eq:2-2-8}, and \eqref{eq:2-2-9}, we
 obtain the desired estimate \eqref{eq:2-2-2}.
\end{proof}
\begin{remark}
 Note that, for arbitrary continuous function
 $\theta:\Omega\times[0,T)\rightarrow\R$,
 \begin{equation*}
  \|\theta\|_{C(\Omega\times[0,T))}
   \geq
   \sup_{x\in\Omega}|\theta(x,0)|,
 \end{equation*}
 hence, in general, we cannot obtain the decay estimate \eqref{eq:2-2-2}, unless $\theta=0$ at
 $t=0$.
\end{remark}
Next, we derive the decay estimate of
$\|\theta\|_{C^{\alpha,\alpha/2}(\Omega\times[0,T))}$ that will be also
used for $\zeta\in X_{M,T}$.
\begin{lemma}
 \label{lem:2-2-2}
 Let arbitrary function $\theta\in C^{2+\alpha,1+\alpha/2}(\Omega\times[0,T))$,
 $\theta(x,0)=0$ for $x\in\Omega$. Then,
 \begin{equation}
  \label{eq:2-2-3}
  \|\theta\|_{C^{\alpha,\alpha/2}(\Omega\times[0,T))}
   \leq
   3 (T+T^{1-\alpha/2})\|\theta\|_{C^{2+\alpha,1+\alpha/2}(\Omega\times[0,T))}.
 \end{equation}
 Thus, for $\zeta\in X_{M,T}$, the estimate~\eqref{eq:2-2-3} holds as well.
\end{lemma}

\begin{proof}
 First we consider $\|\theta\|_{C(\Omega\times[0,T))}$. For
 $x\in\Omega$ and $t\in(0,T)$, we have by $\theta(x,0)=0$,
 \begin{equation}
  \label{eq:2-2-1}
  |\theta(x,t)|
   =
   |\theta(x,t)-\theta(x,0)|
   =
   \left|
    \int_0^t
    \theta_t(x,\tau)
    \,d\tau
   \right|
   \leq
   t\|\theta_t\|_{C(\Omega\times[0,T))}
   ,
 \end{equation}
 thus,
 \begin{equation}
  \label{eq:2-2-10}
  \|\theta\|_{C(\Omega\times[0,T))}
   \leq
   T\|\theta\|_{C^{2+\alpha,1+\alpha/2}(\Omega\times[0,T))}.
 \end{equation}

 Next, we give the estimate of $[\theta]_{\alpha,\Omega\times[0,T)}$. For
 $x,x'\in\Omega$ and $t\in(0,T)$, we first assume $|x-x'|<t^{1/2}$.
 Then, 
 again using the assumption that $\Omega$ is convex,
 the fundamental theorem of calculus and \eqref{eq:2-2-5} lead to,
 \begin{equation*}
  \begin{split}
   |\theta(x,t)-\theta(x',t)|
   &\leq
   |x-x'|\int_0^1
   |\nabla\theta(\tau x+(1-\tau)x',t)|
   \,d\tau
   \\
   &\leq
   T^{(1+\alpha)/2}|x-x'|\|\theta\|_{C^{2+\alpha,1+\alpha/2}(\Omega\times[0,T))}. 
  \end{split}
 \end{equation*}
 Using the assumption $|x-x'|<t^{1/2}$, we have again,
 \begin{equation*}
   |\theta(x,t)-\theta(x',t)|
   \leq T^{(1+\alpha)/2}
   t^{(1-\alpha)/2}
   \|\theta\|_{C^{2+\alpha,1+\alpha/2}(\Omega\times[0,T))}|x-x'|^\alpha
   \leq
   T\|\theta\|_{C^{2+\alpha,1+\alpha/2}(\Omega\times[0,T))}|x-x'|^\alpha.
 \end{equation*}
 Next, we consider the case that $|x-x'|\geq t^{1/2}$. Using the estimate~\eqref{eq:2-2-1}, we have,
 \begin{equation*}
  \begin{split}
   |\theta(x,t)-\theta(x',t)|
   &\leq
   |\theta(x,t)|+|\theta(x',t)| \\
   &\leq
   2t
   \|\theta_t\|_{C(\Omega\times[0,T))}
   \leq
   2 t^{1-\alpha/2}
   \|\theta\|_{C^{2+\alpha,1+\alpha/2}(\Omega\times[0,T))}|x-x'|^\alpha\\
    &\leq
   2 T^{1-\alpha/2}
   \|\theta\|_{C^{2+\alpha,1+\alpha/2}(\Omega\times[0,T))}|x-x'|^\alpha
   .
  \end{split}
 \end{equation*}
 Combining these estimates, we arrive at,
 \begin{equation}
  \label{eq:2-2-11}
  [\theta]_{\alpha,\Omega\times[0,T)}
   \leq
    (T+2T^{1-\alpha/2})\|\theta\|_{C^{2+\alpha,1+\alpha/2}(\Omega\times[0,T))}.
 \end{equation}

 Finally, we consider
 $\langle\theta\rangle_{\alpha/2,\Omega\times[0,T)}$. For $x\in\Omega$ and
 $t,t'\in(0,T)$, the fundamental theorem of calculus leads to,
 \begin{equation*}
  |\theta(x,t)-\theta(x,t')|
   \leq
   \left|
    \int_{t'}^t \theta_t(x,\tau)\,d\tau
   \right|
   \leq |t-t'|
   \|\theta_t\|_{C(\Omega\times[0,T))}
   \leq T^{1-\alpha/2}|t-t'|^{\alpha/2}
   \|\theta_t\|_{C(\Omega\times[0,T))},
 \end{equation*}
 hence,
 \begin{equation}
  \label{eq:2-2-12}
   \langle\theta\rangle_{\alpha/2,\Omega\times[0,T)}
   \leq T^{1-\alpha/2}
   \|\theta\|_{C^{2+\alpha,1+\alpha/2}(\Omega\times[0,T))}.
 \end{equation}
 Combining \eqref{eq:2-2-10}, \eqref{eq:2-2-11}, and \eqref{eq:2-2-12},
 we obtain estimate \eqref{eq:2-2-3}.
\end{proof}

\begin{remark}
 In the proof of the Lemmas above, in order to apply the fundamental theorem of
 calculus, we assumed the sufficient condition on the domain $\Omega$ to be
 convex. However one may generalize the assumptions on the domain to more general
 conditions.
\end{remark}

We later use the norm of the product of the H\"older functions
(cf. \cite[\S 8.5]{MR1406091}). Therefore, we establish the following result.
It is well-known inequalities (for instance, see \cite[\S
4.1]{MR1814364}), but we give a proof for readers convenience. 
\begin{lemma}
\label{lem:2-2-3}
 For functions $\theta \in C^{\alpha,\alpha/2}(\Omega\times[0,T))$ and $\tilde{\theta}\in C^{\alpha,\alpha/2}(\Omega\times[0,T))$,
 the product of  $\theta\tilde{\theta}$ is also in
 $C^{\alpha,\alpha/2}(\Omega\times[0,T))$. Moreover, the following estimate holds,
 \begin{equation*}
  \|\theta\tilde{\theta}\|_{C^{\alpha,\alpha/2}(\Omega\times[0,T))}
   \leq
   \|\theta\|_{C^{\alpha,\alpha/2}(\Omega\times[0,T))}
   \|\tilde{\theta}\|_{C^{\alpha,\alpha/2}(\Omega\times[0,T))}.
 \end{equation*}
\end{lemma}

\begin{proof}
 For $x,x'\in\Omega$, $0< t,t'<T$, we have,
 \begin{equation}
\label{est_prH}
  |\theta(x,t)\tilde{\theta}(x,t)|\leq
   \|\theta\|_{C(\Omega\times[0,T))}
   \|\tilde{\theta}\|_{C(\Omega\times[0,T))}.
 \end{equation}
In addition, we obtain that,
 \begin{equation*}
  \begin{split}
   |\theta(x,t)\tilde{\theta}(x,t)
   -
   \theta(x',t)\tilde{\theta}(x',t)|
   &\leq
   |(\theta(x,t)-\theta(x',t))\tilde{\theta}(x,t)|
   +
   |\theta(x',t)(\tilde{\theta}(x,t)-\tilde{\theta}(x',t))| \\
   &\leq
   \left(
   [\theta]_{\alpha,\Omega\times[0,T)}\|\tilde{\theta}\|_{C(\Omega\times[0,T))}
   +
   \|\theta\|_{C(\Omega\times[0,T))}
   [\tilde{\theta}]_{\alpha,\Omega\times[0,T)}
   \right)
   |x-x'|^\alpha.
  \end{split}
 \end{equation*}
Hence, we have that,
\begin{equation}
\label{est_prH1}
 [\theta \tilde{\theta}]_{\alpha,\Omega\times[0,T)} \leq
   [\theta]_{\alpha,\Omega\times[0,T)}\|\tilde{\theta}\|_{C(\Omega\times[0,T))}
   +
   \|\theta\|_{C(\Omega\times[0,T))}
   [\tilde{\theta}]_{\alpha,\Omega\times[0,T)}.
\end{equation}
Similarly, 
 \begin{equation*}
  \begin{split}
   |\theta(x,t)\tilde{\theta}(x,t)
   -
   \theta(x,t')\tilde{\theta}(x,t')|
   &\leq
   |(\theta(x,t)-\theta(x,t'))\tilde{\theta}(x,t)|
   +
   |\theta(x,t')(\tilde{\theta}(x,t)-\tilde{\theta}(x,t'))| \\
   &\leq
   \left(
   \langle\theta\rangle_{\alpha/2,\Omega\times[0,T)}\|\tilde{\theta}\|_{C(\Omega\times[0,T))}
   +
   \|\theta\|_{C(\Omega\times[0,T))}
   \langle\tilde{\theta}\rangle_{\alpha/2,\Omega\times[0,T)}
   \right)
   |t-t'|^{\alpha/2}.
  \end{split}
 \end{equation*}
Thus, we obtain,
\begin{equation}
\label{est_prH2}
\langle \theta \tilde{\theta}\rangle_{\alpha/2,\Omega\times[0,T)}\leq
   \langle\theta\rangle_{\alpha/2,\Omega\times[0,T)}\|\tilde{\theta}\|_{C(\Omega\times[0,T))}
   +
   \|\theta\|_{C(\Omega\times[0,T))}
   \langle\tilde{\theta}\rangle_{\alpha/2,\Omega\times[0,T)}.
\end{equation}

 Therefore, combining above estimates
 \eqref{est_prH}-\eqref{est_prH2}, we arrive at the desired inequality,
 \begin{equation*}
  \begin{split}
   \|\theta\tilde{\theta}\|_{C^{\alpha,\alpha/2}(\Omega\times[0,T))}
   &=
   \|\theta\tilde{\theta}\|_{C(\Omega\times[0,T))}
   +
   [\theta\tilde{\theta}]_{\alpha,\Omega\times[0,T)}
   +
   \langle\theta\tilde{\theta}\rangle_{\alpha/2,\Omega\times[0,T)} \\
   &\leq
   \|\theta\|_{C(\Omega\times[0,T))}
   \|\tilde{\theta}\|_{C(\Omega\times[0,T))}
   +
   \left(
   [\theta]_{\alpha,\Omega\times[0,T)}\|\tilde{\theta}\|_{C(\Omega\times[0,T))}
   +
   \|\theta\|_{C(\Omega\times[0,T))}
   [\tilde{\theta}]_{\alpha,\Omega\times[0,T)}
   \right) \\
   &\qquad
   +
   \left(
   \langle\theta\rangle_{\alpha/2,\Omega\times[0,T)}\|\tilde{\theta}\|_{C(\Omega\times[0,T))}
   +
   \|\theta\|_{C(\Omega\times[0,T))}
   \langle\tilde{\theta}\rangle_{\alpha/2,\Omega\times[0,T)}
   \right) \\
   &\leq
   \|\theta\|_{C^{\alpha,\alpha/2}(\Omega\times[0,T))}
   \|\tilde{\theta}\|_{C^{\alpha,\alpha/2}(\Omega\times[0,T))}.
  \end{split}
 \end{equation*}
\end{proof}

In this section, results of Lemma~\ref{lem:2-2-1} and Lemma~\ref{lem:2-2-2}
hold for any function $\zeta\in X_{M,T}$. Therefore, we obtained the decay estimates
for the H\"older norms $\|\nabla
\zeta\|_{C^{\alpha,\alpha/2}(\Omega\times[0,T))}$ and
$\|\zeta\|_{C^{\alpha,\alpha/2}(\Omega\times[0,T))}$ of $\zeta\in X_{M,T}$. As a
consequence, in the following sections,  for $\psi\in
X_{M,T}$, the nonlinear term $G(\psi)$ can be treated as a small
perturbation in terms of the H\"older norms.
\subsection{Well-definedness of the solution map}
\label{sec:2-3}
Here, we recall the function space $X_{M,T}$ defined in
\eqref{eq:2-3-9}. Here, for $\psi\in
X_{M,T}$, our goal is to consider first the linear parabolic equation \eqref{eq:2-3-8}
associated with the nonlinear problem \eqref{eq:2-1-4}. We also recall the
definition of the solution map $A:\psi\mapsto\eta$ from the
Definition~\ref{def:solmap} associated with the linear parabolic model \eqref{eq:2-3-8}.
Therefore, 
in this Section~\ref{sec:2-3} and in the next Section~\ref{sec:2-4}, we are going to show that the
solution map $A:\psi\mapsto\eta$ is a contraction mapping on $X_{M,T}$,
where $\eta$ is a solution of \eqref{eq:2-3-8}. Once we will show that the
solution map $A$ is a contraction, we can obtain a fixed point $\xi\in
X_{M,T}$ for the solution map $A$,  and hence $\xi$ will be a solution of
\eqref{eq:2-1-4},  \cite[\S 7.2]{MR2759829}.
\par First, we will show that the solution map is well-defined on
$X_{M,T}$, namely that there exist appropriate positive constants $M,T>0$ such that for any
$\psi\in X_{M,T}$, solution $\eta=A\psi$ of the linear parabolic equation
\eqref{eq:2-3-8} belongs to $X_{M,T}$.

Let us now recall the Schauder estimates for the following linear
parabolic equation:
\begin{equation}
 \label{eq:2-3-1}
 \left\{
  \begin{aligned}
   \frac{\partial w}{\partial t}
   &=
   Lw+g(x,t),\quad
   x\in\Omega,\
   t>0, \\
   \nabla
   w
   \cdot
   \nu
   \bigg|_{\partial\Omega}
   &=
   0,\quad
   t>0, \\
   w(0,x)
   &=
   0,
   \quad
   x\in\Omega.
  \end{aligned}
 \right.
\end{equation}
Here, the operator $L$ is defined in \eqref{eq:2-1-1}. The following Schauder estimates
for the solution of \eqref{eq:2-3-1} can be applicable.
\begin{proposition}%
 [{\cite[Theorem 5.3 in Chapter IV]{MR0241822}, \cite[Theorem
 4.31]{MR1465184}}] \label{prop:2-3-1}
 Assume the strong positivity \eqref{eq:2-0-1}, the
 regularity \eqref{eq:2-0-2}, and let $L$ be the differential operator
 defined in \eqref{eq:2-1-1}.
 For any H\"older continuous function $g\in
 C^{\alpha,\alpha/2}(\Omega\times[0,T))$, there uniquely exists a
 solution $w \in C^{2+\alpha,1+\alpha/2}(\Omega\times[0,T))$ of
 \eqref{eq:2-3-1},  such that,
 \begin{equation}
  \label{eq:2-3-2}
  \|w\|_{C^{2+\alpha,1+\alpha/2}(\Omega\times[0,T))}
   \leq
   \Cl{const:2.2}\|g\|_{C^{\alpha,\alpha/2}(\Omega\times[0,T))},
 \end{equation}
 where $\Cr{const:2.2}>0$ is a positive constant.
\end{proposition}
Using the Schauder estimate \eqref{eq:2-3-2}, we now show the well-definedness
of the solution map $A$ in $X_{M,T}$.

\begin{lemma}
 \label{lem:2-3-1}
 Assume the strong positivity \eqref{eq:2-0-1}, the
 regularity \eqref{eq:2-0-2}, and let $L$ be the differential operator
 defined in \eqref{eq:2-1-1}. Then,  there are constants $M>0$ and
 $T_0>0$, such that for $0<T\leq T_0$ and $\psi\in X_{M,T}$, the image
 of the solution map $A\psi$ belongs to $X_{M,T}$ and the map $A$ is
 well-defined on $X_{M,T}$.
\end{lemma}

\begin{proof}
 Let us assume that we have constants $M,T>0$ that will be defined
 later, then consider $\psi\in
 X_{M,T}$.  We use the Schauder estimate \eqref{eq:2-3-2} for $L$ and
 for $g = g_0 + G(\psi)$, where $L$, $G(\psi)$ and $g_0$ are defined as
 in \eqref{eq:2-1-1}.
 First, we note that from the strong positivity
 \eqref{eq:2-0-1} and the regularity \eqref{eq:2-0-2},  there is a
 positive constant $\Cl{const:2.3}>0$ which depends only on
 $\|b\|_{C^{1+\alpha,(1+\alpha)/2}(\Omega\times[0,T))}$,
 $\|D\|_{C^{1+\alpha}(\Omega)}$, $\|\phi\|_{C^{1+\alpha}(\Omega)}$,
 $\|h_0\|_{C^{2+\alpha}(\Omega)}$, and
 the constant $\Cr{const:2.8}$ in \eqref{eq:2-0-1} such that,
 \begin{equation}
  \label{eq:2-3-10}
  \|g_0\|_{C^{\alpha,\alpha/2}(\Omega\times[0,T))}
   \leq
   \Cr{const:2.3}.
 \end{equation}
 Next, we calculate the norm of $\frac{(b(x,t))^2}{2D(x)}
 f^{\mathrm{eq}}(x) |\nabla \psi|^2$. Using Lemma \ref{lem:2-2-3}, 
 the strong positivity \eqref{eq:2-0-1} and the
 regularity \eqref{eq:2-0-2}, we obtain for
 $\psi\in X_{M,T}$,
 \begin{equation*}
  \left\|
   \frac{(b(x,t))^2}{2D(x)}
   f^{\mathrm{eq}}(x)
   |\nabla \psi|^2
  \right\|_{C^{\alpha,\alpha/2}(\Omega\times[0,T))}
  \leq
  \left\|
   \frac{(b(x,t))^2}{2D(x)}
   f^{\mathrm{eq}}(x)
  \right\|_{C^{\alpha,\alpha/2}(\Omega\times[0,T))}
  \left\|
   \nabla \psi
  \right\|^2_{C^{\alpha,\alpha/2}(\Omega\times[0,T))}.
 \end{equation*}
Noting that $\psi(x,0)=0$ for $x\in\Omega$, we can apply Lemma
 \ref{lem:2-2-1} and use the decay estimate \eqref{eq:2-2-2} to show that,
 \begin{equation*}
  \left\|
   \nabla \psi
  \right\|^2_{C^{\alpha,\alpha/2}(\Omega\times[0,T))}
  \leq
  9\|\psi\|^2_{C^{2+\alpha,1+\alpha/2}(\Omega\times[0,T)))}
   (T^{(1+\alpha)/2}+T^{1/2})^2.
 \end{equation*}
 Since $\psi\in X_{M,T}$,
 $\|\psi\|_{C^{2+\alpha,1+\alpha/2}(\Omega\times[0,T)))}\leq M$, hence
 we have,
 \begin{equation}
  \label{eq:2-3-3}
   \left\|
    \frac{(b(x,t))^2}{2D(x)}
    f^{\mathrm{eq}}(x)
    |\nabla \psi|^2
   \right\|_{C^{\alpha,\alpha/2}(\Omega\times[0,T))}
   \leq
   9
   \left\|
    \frac{(b(x,t))^2}{2D(x)}
    f^{\mathrm{eq}}(x)
   \right\|_{C^{\alpha,\alpha/2}(\Omega\times[0,T))}
   M^2
   (T^{(1+\alpha)/2}+T^{1/2})^2.
 \end{equation}
 Next, we calculate the norm of $\frac{(b(x,t))^2}{2(D(x))^2}
 f^{\mathrm{eq}}(x)
 \psi
 \nabla \psi
 \cdot
 \nabla
 D(x)$. Using Lemma \ref{lem:2-2-3}, the strong
 positivity \eqref{eq:2-0-1} and the regularity \eqref{eq:2-0-2}, we estimate,
 \begin{equation*}
  \begin{split}
   \left\|
   \frac{(b(x,t))^2}{2(D(x))^2}
   f^{\mathrm{eq}}(x)
   \psi
   \nabla \psi
   \cdot
   \nabla
   D(x)
   \right\|_{C^{\alpha,\alpha/2}(\Omega\times[0,T))}
   &\leq
   \left\|
   \frac{(b(x,t))^2}{2(D(x))^2}
   f^{\mathrm{eq}}(x)
   \nabla
   D(x)
   \right\|_{C^{\alpha,\alpha/2}(\Omega\times[0,T))} \\
   &\qquad
   \times\left\|
   \psi
   \right\|_{C^{\alpha,\alpha/2}(\Omega\times[0,T))}
   \left\|
   \nabla \psi
   \right\|_{C^{\alpha,\alpha/2}(\Omega\times[0,T))}
   .
  \end{split}
 \end{equation*}
 Using Lemma \ref{lem:2-2-1} and \ref{lem:2-2-2} with the initial
 condition $\psi=0$ at $t=0$, we have by \eqref{eq:2-2-2} and
 \eqref{eq:2-2-3} that,
 \begin{equation*}
  \left\|
   \nabla \psi
  \right\|_{C^{\alpha,\alpha/2}(\Omega\times[0,T))}
  \leq
  3
  \|\psi\|_{C^{2+\alpha,1+\alpha/2}(\Omega\times[0,T)))}
   (T^{(1+\alpha)/2}+T^{1/2}),
 \end{equation*}
 and,
 \begin{equation*}
  \|\psi\|_{C^{\alpha,\alpha/2}(\Omega\times[0,T))}
   \leq
   3
   \|\psi\|_{C^{2+\alpha,1+\alpha/2}(\Omega\times[0,T))}
   (T+T^{1-\alpha/2}).
 \end{equation*}
 Again, since $\psi\in X_{M,T}$,
 $\|\psi\|_{C^{2+\alpha,1+\alpha/2}(\Omega\times[0,T)))}\leq M$, and
 thus, we obtain,
  \begin{equation}
   \label{eq:2-3-4}
   \begin{split}
   \left\|
    \frac{(b(x,t))^2}{2(D(x))^2}
    f^{\mathrm{eq}}(x)
    \psi
    \nabla \psi
    \cdot
    \nabla
    D(x)
    \right\|_{C^{\alpha,\alpha/2}(\Omega\times[0,T))}
    &\leq
    9
    \left\|
    \frac{(b(x,t))^2}{2(D(x))^2}
    f^{\mathrm{eq}}(x)
    \nabla
    D(x)
    \right\|_{C^{\alpha,\alpha/2}(\Omega\times[0,T))} \\
    &\qquad
    \times
    M^2
    (T^{(1+\alpha)/2}+T^{1/2})
    (T+T^{1-\alpha/2}).
   \end{split}
 \end{equation}
 Together with \eqref{eq:2-3-3} and \eqref{eq:2-3-4}, we can take a
 positive constant $\Cl{const:2.4}>0$ which depends only on
 $\|b\|_{C^{1+\alpha,(1+\alpha)/2}(\Omega\times[0,T))}$,
 $\|D\|_{C^{1+\alpha}(\Omega)}$, $\|\phi\|_{C^{1+\alpha}(\Omega)}$,
 and the constant $\Cr{const:2.8}$,  such that,
 \begin{equation}
  \label{eq:2-3-11}
   \left\|
    G(\psi)
   \right\|_{C^{\alpha,\alpha/2}(\Omega\times[0,T))}
   \leq
   \Cr{const:2.4}M^2
   \kappa(T),
 \end{equation}
 where
 \begin{equation}
  \label{eq:2-3-7}
  \kappa(T)
   =
   (T^{(1+\alpha)/2}+T^{1/2})^2
   +
   (T^{(1+\alpha)/2}+T^{1/2})
   (T+T^{1-\alpha/2}).
 \end{equation}
 Note that $\kappa(T)$ is an increasing function with respect to $T>0$ and
 $\kappa(T)\rightarrow0$ as $T\downarrow0$. By the Schauder estimate
 \eqref{eq:2-3-2}, together with \eqref{eq:2-3-10} and
 \eqref{eq:2-3-11}, the solution $\xi=A\psi$ of the linear parabolic
 equation \eqref{eq:2-3-8} satisfies,
 \begin{equation}
  \label{eq:2-3-5}
  \|A\psi\|_{C^{2+\alpha,1+\alpha/2}(\Omega\times[0,T))}
   \leq
   \Cr{const:2.2}
   \left(
    \Cr{const:2.3}
    +
    \Cr{const:2.4}M^2\kappa(T)
   \right).
 \end{equation}

 In order to guarantee
 $\|A\psi\|_{C^{2+\alpha,1+\alpha/2}(\Omega\times[0,T))}\leq M$ for
 $0<T\leq T_0$, we take,
 \begin{equation}
  \label{eq:2-3-6}
  M:=2 \Cr{const:2.2}\Cr{const:2.3},
   \qquad
   \Cr{const:2.4}M^2
   \kappa(T_0)
   \leq
   \Cr{const:2.3}.
 \end{equation}
 Then from \eqref{eq:2-3-5}, $
 \|A\psi\|_{C^{2+\alpha,1+\alpha/2}(\Omega\times[0,T))} \leq M$ for
 $0<T\leq T_0$,  hence $A\psi \in X_{M,T}$.
\end{proof}

\begin{remark}
 Note that from \eqref{eq:2-3-6}, a positive constant $M>0$ depends on
 $\|b\|_{C^{1+\alpha,(1+\alpha)/2}(\Omega\times[0,T))}$,
 $\|D\|_{C^{1+\alpha}(\Omega)}$, $\|\phi\|_{C^{1+\alpha}(\Omega)}$,
 $\|h_0\|_{C^{2+\alpha}(\Omega)}$, and the constant
 $\Cr{const:2.8}$. Also,  from \eqref{eq:2-3-6}, a time interval $T_0>0$
 can be estimated as,
 \begin{equation}
  \label{eq:2-3-12}
  \kappa(T_0)
   \leq
   \frac{1}{4\Cr{const:2.2}^2\Cr{const:2.3}\Cr{const:2.4}}.
  \end{equation}
 Since $\psi=0$ at $t=0$, the auxiliary function $\kappa(T)$  can be written
 explicitly as in \eqref{eq:2-3-7}, in order to
 estimate the H\"older norm of nonlinear term $G(\psi)$. Thus, using \eqref{eq:2-3-12}, we
 obtain the explicit estimate of the time-interval $T_0>0$ to ensure
 that the solution map $A$ is well-defined on $X_{M,T}$.
\end{remark}

\subsection{The contraction property}
\label{sec:2-4}
In this section, we show that the solution map
$A:X_{M,T}\ni\psi\mapsto\eta\in X_{M,T}$, where $\eta$ is a solution of
\eqref{eq:2-3-8}, is contraction on $X_{M,T}$. The explicit decay
estimates for the H\"older norm of $\psi\in X_{M,T}$ obtained in Lemmas
\ref{lem:2-2-1} and \ref{lem:2-2-2}, are essential for the derivation of the
smallness of the nonlinear term $G(\psi)$. Because, for $\psi\in
X_{M,T}$,  H\"older norms
$\|\nabla \psi\|_{C^{\alpha,\alpha/2}(\Omega\times[0,T))}$ and
$\|\psi\|_{C^{\alpha,\alpha/2}(\Omega\times[0,T))}$ continuously go to
$0$ as $T\rightarrow0$, thus,  the Lipschitz constant of $A$ in
$C^{2+\alpha,1+\alpha/2}(\Omega\times[0,T))$ can be
taken smaller than $1$ if $T$ is sufficiently
small. This is the reason why we consider the change of variables
\eqref{eq:2-1-2}, and as result, consider the zero initial value problem
\eqref{eq:2-1-4} subject to the homogeneous Neumann boundary
condition.

\begin{lemma}
 \label{lem:2-4-1}
Assume the strong positivity \eqref{eq:2-0-1},
 regularity \eqref{eq:2-0-2}, and let $L$ be the differential operator
 defined in \eqref{eq:2-1-1}. 
 Let $M>0$ and $T_0>0$ be the constants obtained in
 Lemma \ref{lem:2-3-1}, \eqref{eq:2-3-6}. Then, there exists $T_1 \in
 (0,T_0]$ such that $A$ is contraction on $X_{M,T}$ for $0<T\leq T_1$.
\end{lemma}

\begin{proof}[Proof of Lemma \ref{lem:2-4-1}]
 We take $0<T\leq T_0$, where $T$ will be specified later in the proof. For
 $\psi_1$, $\psi_2\in X_{M,T}$, let
 $\tilde{\eta}:=A\psi_1-A\psi_2$. Then from \eqref{eq:2-3-8},  $\tilde{\eta}$ satisfies,
\begin{equation}
 \label{eq:2-4-11}
 \left\{
  \begin{aligned}
   \frac{\partial \tilde{\eta}}{\partial t}
   &=
   L\tilde{\eta}+G(\psi_1)-G(\psi_2),\quad
   x\in\Omega,\
   t>0, \\
   \nabla
   \tilde{\eta}
   \cdot
   \nu
   \bigg|_{\partial\Omega}
   &=
   0,\quad
   t>0, \\
   \tilde{\eta}(0,x)
   &=
   0,
   \quad
   x\in\Omega.
  \end{aligned}
 \right.
\end{equation}
 Due to zero Neumann boundary and the initial conditions for
 $\tilde{\eta}$, we can use the Schauder estimate \eqref{eq:2-3-2} for
 the system
 \eqref{eq:2-4-11}, hence, we have,
 \begin{equation}
  \label{eq:2-4-3}
  \|\tilde{\eta}\|_{C^{2+\alpha,1+\alpha/2}(\Omega\times[0,T))}
   \leq
   \Cr{const:2.2}
   \|
   G(\psi_1)-G(\psi_2)
   \|_{C^{\alpha,\alpha/2}(\Omega\times[0,T))}.
 \end{equation}
 By direct calculation of the difference of the nonlinear terms
 $G(\psi)$ \eqref{eq:2-1-1}, we have,
 \begin{equation}
  \label{eq:2-4-10}
  G(\psi_1)-G(\psi_2)
   =
   \frac{(b(x,t))^2}{2D(x)}
   f^{\mathrm{eq}}(x)
   (|\nabla \psi_1|^2-|\nabla \psi_2|^2)
   -
   \frac{(b(x,t))^2}{2(D(x))^2}
   f^{\mathrm{eq}}(x)
   (
   \psi_1
   \nabla \psi_1
   -
   \psi_2
   \nabla \psi_2
   )
   \cdot
   \nabla
   D(x).
 \end{equation}

 First, we estimate $\|\frac{(b(x,t))^2}{2D(x)}
 f^{\mathrm{eq}}(x) (|\nabla \psi_1|^2-|\nabla
 \psi_2|^2)\|_{C^{\alpha,\alpha/2}(\Omega\times[0,T))}$. Since,
 \begin{equation*}
  \left|
   |\nabla \psi_1|^2-|\nabla \psi_2|^2
   \right|
   =
  \left|
   (\nabla \psi_1+\nabla \psi_2)
   \cdot
   (\nabla \psi_1-\nabla \psi_2)
   \right|,
 \end{equation*}
 we have due to Lemma \ref{lem:2-2-3} that,
 \begin{equation}
  \label{eq:2-4-5}
  \|
   |\nabla \psi_1|^2-|\nabla \psi_2|^2
   \|_{C^{\alpha,\alpha/2}(\Omega\times[0,T))}
   \leq
   \|
   \nabla \psi_1+\nabla \psi_2
   \|_{C^{\alpha,\alpha/2}(\Omega\times[0,T))}
   \|
   \nabla \psi_1-\nabla \psi_2
   \|_{C^{\alpha,\alpha/2}(\Omega\times[0,T))}.
 \end{equation}
 Since $\psi_1$, $\psi_2 \in X_{M,T}$, we have that $\psi_1-\psi_2=0$
 at $t=0$, and Lemma
 \ref{lem:2-2-1} is applicable here to functions $\psi_1, \psi_2$ and $\psi_1-\psi_2$,
\begin{equation}
  \label{eq:2-4-6}
  \begin{split}
   \|\nabla\psi_1\|_{C^{\alpha,\alpha/2}(\Omega\times[0,T))}
   &\leq
   3  (T^{(1+\alpha)/2}+T^{1/2})\|\psi_1\|_{C^{2+\alpha,1+\alpha/2}(\Omega\times[0,T))}, \\
   \|\nabla\psi_2\|_{C^{\alpha,\alpha/2}(\Omega\times[0,T))}
   &\leq
   3 (T^{(1+\alpha)/2}+T^{1/2})\|\psi_2\|_{C^{2+\alpha,1+\alpha/2}(\Omega\times[0,T))}, \\
   \|\nabla\psi_1-\nabla\psi_2\|_{C^{\alpha,\alpha/2}(\Omega\times[0,T))}
   &\leq
   3 (T^{(1+\alpha)/2}+T^{1/2})\|\psi_1-\psi_2\|_{C^{2+\alpha,1+\alpha/2}(\Omega\times[0,T))}.
  \end{split}
\end{equation}
 Combining estimates \eqref{eq:2-4-5} and \eqref{eq:2-4-6}, we obtain,
 \begin{equation*}
  \begin{split}
   &\quad
   \|
   |\nabla \psi_1|^2-|\nabla \psi_2|^2
   \|_{C^{\alpha,\alpha/2}(\Omega\times[0,T))} \\
   &\leq
   9  (T^{(1+\alpha)/2}+T^{1/2})^2
   (
   \|
   \psi_1
   \|_{C^{2+\alpha,1+\alpha/2}(\Omega\times[0,T))}
   +
   \|
   \psi_2
   \|_{C^{2+\alpha,1+\alpha/2}(\Omega\times[0,T))}
   )
   \|
   \psi_1-\psi_2
   \|_{C^{2+\alpha,1+\alpha/2}(\Omega\times[0,T))}.
  \end{split}
 \end{equation*}
 Therefore, using the strong positivity \eqref{eq:2-0-1},
 the regularity \eqref{eq:2-0-2}, and that functions $\psi_1,\psi_2\in X_{M,T}$, we
 arrive at the inequality,
 \begin{equation}
  \label{eq:2-4-1}
   \left\|
   \frac{(b(x,t))^2}{2D(x)}
   f^{\mathrm{eq}}(x) (|\nabla \psi_1|^2-|\nabla
   \psi_2|^2)
   \right\|_{C^{\alpha,\alpha/2}(\Omega\times[0,T))}
   \leq
   \Cl{const:2.5}M
   (T^{(1+\alpha)/2}+T^{1/2})^2
   \|
   \psi_1-\psi_2
   \|_{C^{2+\alpha,1+\alpha/2}(\Omega\times[0,T))}.
 \end{equation}
Here, constant 
 \begin{equation*}
  \Cr{const:2.5}
   =
   9
   \left\|
    \frac{(b(x,t))^2}{D(x)}
    f^{\mathrm{eq}}(x)
   \right\|_{C^{\alpha,\alpha/2}(\Omega\times[0,T))}
 \end{equation*}
 is a positive constant which depends only on
 $\|b\|_{C^{\alpha,\alpha/2}(\Omega\times[0,T))}$,
 $\|D\|_{C^{\alpha}(\Omega)}$, $\|\phi\|_{C^{\alpha}(\Omega)}$,
 and the constant $\Cr{const:2.8}$ in \eqref{eq:2-0-1}.

 Next, we estimate,
 $
 \|\frac{(b(x,t))^2}{2(D(x))^2}
 f^{\mathrm{eq}}(x)
 (
 \psi_1
 \nabla \psi_1
 -
 \psi_2
 \nabla \psi_2
 )
 \cdot
 \nabla
 D(x)
 \|_{C^{\alpha,\alpha/2}(\Omega\times[0,T))}$. Since, we can write,
 \begin{equation*}
  \psi_1 \nabla \psi_1
   -
   \psi_2 \nabla \psi_2
   =
   \psi_1
   (
   \nabla \psi_1
   -
   \nabla \psi_2
   )
   +
   (
   \psi_1 - \psi_2
   )
   \nabla\psi_2,
 \end{equation*}
  we can use Lemma \ref{lem:2-2-3} again,
 \begin{multline}
  \label{eq:2-4-7}
  \|
   \psi_1 \nabla \psi_1
   -
   \psi_2 \nabla \psi_2
   \|_{C^{\alpha,\alpha/2}(\Omega\times[0,T))} \\
   \leq
   \|
   \psi_1
   \|_{C^{\alpha,\alpha/2}(\Omega\times[0,T))}
   \|
   \nabla \psi_1-\nabla \psi_2
   \|_{C^{\alpha,\alpha/2}(\Omega\times[0,T))}
   +
   \|
   \nabla \psi_2
   \|_{C^{\alpha,\alpha/2}(\Omega\times[0,T))}
   \|
   \psi_1-\psi_2
   \|_{C^{\alpha,\alpha/2}(\Omega\times[0,T))}.
 \end{multline}
 Since $\psi_1$, $\psi_2 \in X_{M,T}$, we have that $\psi_1-\psi_2=0$
 at $t=0$, and thus,  we can use Lemma
 \ref{lem:2-2-1} and \ref{lem:2-2-2} to obtain,
 \begin{equation}
  \label{eq:2-4-8}
  \begin{split}
   \|\psi_1\|_{C^{\alpha,\alpha/2}(\Omega\times[0,T))}
   &\leq
   3  (T+T^{1-\alpha/2})\|\psi_1\|_{C^{2+\alpha,1+\alpha/2}(\Omega\times[0,T))}, \\
   \|\nabla\psi_2\|_{C^{\alpha,\alpha/2}(\Omega\times[0,T))}
   &\leq
   3  (T^{(1+\alpha)/2}+T^{1/2})\|\psi_2\|_{C^{2+\alpha,1+\alpha/2}(\Omega\times[0,T))}, \\
   \|\psi_1-\psi_2\|_{C^{\alpha,\alpha/2}(\Omega\times[0,T))}
   &\leq
   3  (T+T^{1-\alpha/2})\|\psi_1-\psi_2\|_{C^{2+\alpha,1+\alpha/2}(\Omega\times[0,T))}, \\
   \|\nabla\psi_1-\nabla\psi_2\|_{C^{\alpha,\alpha/2}(\Omega\times[0,T))}
   &\leq
   3 (T^{(1+\alpha)/2}+T^{1/2})\|\psi_1-\psi_2\|_{C^{2+\alpha,1+\alpha/2}(\Omega\times[0,T))}.
  \end{split}
 \end{equation}
 Combining \eqref{eq:2-4-7} and \eqref{eq:2-4-8},
 we obtain the estimate,
 \begin{equation*}
  \begin{split}
   &\quad
   \|
   \psi_1 \nabla \psi_1
   -
   \psi_2 \nabla \psi_2
   \|_{C^{\alpha,\alpha/2}(\Omega\times[0,T))} \\
   &\leq
   9 (T^{(1+\alpha)/2}+T^{1/2})(T+T^{1-\alpha/2}) (
   \|
   \psi_1
   \|_{C^{2+\alpha,1+\alpha/2}(\Omega\times[0,T))}
   +
   \|
   \psi_2
   \|_{C^{2+\alpha,1+\alpha/2}(\Omega\times[0,T))}
   )\\
   &\quad
   \times
   \|
   \psi_1-\psi_2
   \|_{C^{2+\alpha,1+\alpha/2}(\Omega\times[0,T))}.
  \end{split}
 \end{equation*}
 Therefore, using the strong positivity
 \eqref{eq:2-0-1}, the regularity \eqref{eq:2-0-2},  and that $\psi_1,\psi_2\in
 X_{M,T}$, we get,
 \begin{equation}
  \label{eq:2-4-2}
  \begin{split}
   &\quad
   \left\|
   \frac{(b(x,t))^2}{2(D(x))^2}
   f^{\mathrm{eq}}(x)
   (
   \psi_1
   \nabla \psi_1
   -
   \psi_2
   \nabla \psi_2
   )
   \cdot
   \nabla
   D(x)
   \right\|_{C^{\alpha,\alpha/2}(\Omega\times[0,T))} \\
   &\leq
   \Cl{const:2.6}M
   (T^{(1+\alpha)/2}+T^{1/2})(T+T^{1-\alpha/2})
   \|
   \psi_1-\psi_2
   \|_{C^{2+\alpha,1+\alpha/2}(\Omega\times[0,T))},
  \end{split}
 \end{equation}
 where constant
 \begin{equation*}
  \Cr{const:2.6}
   =
   9
   \left\|
    \frac{(b(x,t))^2}{(D(x))^2}
    f^{\mathrm{eq}}(x)
    \nabla D(x)
   \right\|_{C^{\alpha,\alpha/2}(\Omega\times[0,T))}
 \end{equation*}
 is a positive constant which depends only on
 $\|b\|_{C^{\alpha,\alpha/2}(\Omega\times[0,T))}$,
 $\|D\|_{C^{1+\alpha}(\Omega)}$, $\|\phi\|_{C^{\alpha}(\Omega)}$, and
 the constant $\Cr{const:2.8}$ in \eqref{eq:2-0-1}.

 Finally, combining \eqref{eq:2-4-3}, \eqref{eq:2-4-10}, \eqref{eq:2-4-1}, and
 \eqref{eq:2-4-2}, we arrive at the estimate,
 \begin{equation*}
  \begin{split}
   \|A\psi_1-A\psi_2\|_{C^{2+\alpha,1+\alpha/2}(\Omega\times[0,T))}
   &=
   \|\tilde{\eta}\|_{C^{2+\alpha,1+\alpha/2}(\Omega\times[0,T))} \\
   &\leq
   \Cl{const:2.7}
   M\kappa(T)
   \|\psi_1-\psi_2\|_{C^{2+\alpha,1+\alpha/2}(\Omega\times[0,T))},
  \end{split}
 \end{equation*}
 where
 $\Cr{const:2.7}=\Cr{const:2.2}\max\{\Cr{const:2.5},\
 \Cr{const:2.6}\}>0$ is a positive constant and,
 \begin{equation}
  \label{eq:2-4-9}
  \kappa(T)
   =
   (T^{(1+\alpha)/2}+T^{1/2})^2
   +
   (T^{(1+\alpha)/2}+T^{1/2})
   (T+T^{1-\alpha/2}).
 \end{equation}
 Note that $\kappa(T)$ is increasing with respect to $T>0$ and
 $\kappa(T)\rightarrow0$ as $T\downarrow0$. Taking $T_1\in (0,T_0]$ such
 that,
 \begin{equation}
  \label{eq:2-4-4}
  \Cr{const:2.7} M\kappa(T_1)< 1,
 \end{equation}
 the solution map $A$ is a contraction mapping on $X_{M,T}$ for $0<T\leq
 T_1$.
\end{proof}

\begin{remark}
 \label{rem:2-1-1}
 Note that, for $h\in C^{2+\alpha,1+\alpha/2}(\Omega\times[0,T))$,
 $\|h\|_{C^{\alpha,\alpha/2}(\Omega\times[0,T))}$ and $\|\nabla
 h\|_{C^{\alpha,\alpha/2}(\Omega\times[0,T))}$ do not vanish as
 $T\downarrow0$ in general. On the other hand, when $\psi=0$ at $t=0$,
 H\"older's norms $\|\psi\|_{C^{\alpha,\alpha/2}(\Omega\times[0,T))}$ and
 $\|\nabla \psi\|_{C^{\alpha,\alpha/2}(\Omega\times[0,T))}$ continuously
 go to $0$ as $T\downarrow0$ by \eqref{eq:2-2-2} and \eqref{eq:2-2-3}.
 Thus,  we derived the explicit time-interval estimates in \eqref{eq:2-4-9}
 and in \eqref{eq:2-4-4}, to ensure that the solution map $A$ is a
 contraction map.

 Further note that, we may show directly the well-definedness and contraction for
 the solution map associated with the problem \eqref{eq:2-1-5}. Still it
 is worth considering variable $\xi$ in \eqref{eq:2-1-2}: we can easily
 construct a contraction mapping $A$ on $X_{M,T}$ and get the estimates
 \eqref{eq:2-3-12} and \eqref{eq:2-4-4} to guarantee the
 well-definedness and contraction for the solution map.
\end{remark}

We are now in position to prove existence of a solution of
\eqref{eq:2-0-10}.

\begin{proof}[Proof of Theorem \ref{thm:2-0-1}]
 Let $M>0$ be a positive constant obtained in Lemma \ref{lem:2-3-1},
 \eqref{eq:2-3-6}, and let $T_1>0$ be a positive constant from
 Lemma \ref{lem:2-4-1}, \eqref{eq:2-4-4}.  Then,  due to Lemma
 $\ref{lem:2-3-1}$ and \ref{lem:2-4-1}, the solution map $A$ is
 a contraction on $X_{M,T_1}$. Therefore,  there is a fixed point $\xi\in
 X_{M,T_1}$,  such that $\xi=A\xi$ and $\xi$ is a classical solution of
 \eqref{eq:2-1-4}. Thus,
 \begin{equation*}
  \rho(x,t)
   =
   \exp
   \left(
    \frac{\xi(x,t)+h_0(x)}{D(x)}
   \right)
 \end{equation*}
 is a classical solution of \eqref{eq:2-0-10}.
\end{proof}

In this section, we constructed a solution $\rho$ using auxiliary variables
$h$ in \eqref{eq:2-1-3} and $\xi$ in \eqref{eq:2-1-2}. Since $\xi=0$ at $t=0$, the time interval of a solution
can be explicitly estimated as in \eqref{eq:2-3-6} and in
\eqref{eq:2-4-4}. As a last step of our construction, we
will show uniqueness of the solution $\rho$ of \eqref{eq:2-0-10} in the
next section.
\section{Uniqueness}
\label{sec:3}

In this section, we show uniqueness for a local solution of
\eqref{eq:2-0-4}. As in Section \ref{sec:2}, uniqueness of a solution
of \eqref{eq:2-0-10} implies the uniqueness of a solution to \eqref{eq:2-0-4}. We make the same
assumptions as we did to show existence of a classical solution of
\eqref{eq:2-0-10}. 
Note that, the contraction property of the solution map $A$ implies
the uniqueness of the fixed point on $X_{M,T}$, but not on $C^{2+\alpha,1+\alpha/2}(\Omega\times[0,T))$. Nevertheless, similar to the proof of the contraction
property of the solution map $A$, Lemma~\ref{lem:2-4-1} in Section
\ref{sec:2}, we show below
uniqueness for a classical solution of \eqref{eq:2-0-10} on $C^{2+\alpha,1+\alpha/2}(\Omega\times[0,T))$.

\begin{theorem}
 \label{thm:3-0-1}
 Let $b(x,t)$, $\phi(x)$, $D(x)$, $\rho_0(x)$ and $\Omega$ satisfy the
 strong positivity \eqref{eq:2-0-1}, the H\"older regularity
 \eqref{eq:2-0-2} for $0<\alpha<1$, and the compatibility for the
 initial data \eqref{eq:2-0-3}, respectively. Then, there exists $T>0$ such that, if
 $\rho_1$, $\rho_2\in C^{2+\alpha,1+\alpha/2}(\Omega\times[0,T))$ are
 classical solutions of \eqref{eq:2-0-10}, then $\rho_1=\rho_2$ on
 $\Omega\times[0,T)$.
\end{theorem}

\begin{proof}
 First,  note that from Lemma \ref{lem:2-1-2} and Lemma \ref{lem:2-1-1},
 it is sufficient to show uniqueness for a solution of \eqref{eq:2-1-4}.
 Hereafter,  we will show the uniqueness for a classical solution of the
 problem \eqref{eq:2-1-4}.

 Let $\xi_1, \xi_2\in C^{2+\alpha,1+\alpha/2}(\Omega\times[0,T))$ be
 two distinct solutions of \eqref{eq:2-1-4}. We will prove that $\xi_1=\xi_2$ in
 $\Omega\times[0,T)$ for sufficiently small $T>0$ using
 contradiction argument. Assume that $\xi_1$ and $\xi_2$ are two
 distinct solutions  in $\Omega\times[0,T)$ for any
 $T>0$. Then,  subtracting $\xi_1$ from $\xi_2$, we obtain the equation,
 \begin{equation*}
  \frac{\partial (\xi_1-\xi_2)}{\partial t}
   =
   L(\xi_1-\xi_2)+G(\xi_1)-G(\xi_2),
 \end{equation*}
 where $L$ and $G$ are defined in \eqref{eq:2-1-1}. Since
 $\xi_1-\xi_2=0$ at $t=0$, we can apply the Schauder estimates
 \eqref{eq:2-3-2}, and we obtain,
 \begin{equation}
  \label{eq:3-0-6}
  \|\xi_1-\xi_2\|_{C^{2+\alpha,1+\alpha/2}(\Omega\times[0,T))}
   \leq
   \Cr{const:2.2}
   \|
   G(\xi_1)-G(\xi_2)
   \|_{C^{\alpha,\alpha/2}(\Omega\times[0,T))}.
 \end{equation}

 As in the proof of the Lemma~\ref{lem:2-4-1}, we estimate the norm of,
 \begin{equation}
  \label{eq:3-0-1}
   G(\xi_1)-G(\xi_2)
   =
   \frac{(b(x,t))^2}{2D(x)}
   f^{\mathrm{eq}}(x)
   (|\nabla \xi_1|^2-|\nabla \xi_2|^2)
   -
   \frac{(b(x,t))^2}{2(D(x))^2}
   f^{\mathrm{eq}}(x)
   (
   \xi_1
   \nabla \xi_1
   -
   \xi_2
   \nabla \xi_2
   )
   \cdot
   \nabla
   D(x).
 \end{equation}
 Let $M(T):=\max\{
 \|\xi_1\|_{C^{2+\alpha,1+\alpha/2}(\Omega\times[0,T))},
 \|\xi_2\|_{C^{2+\alpha,1+\alpha/2}(\Omega\times[0,T))}\}>0$.  Then,
 $\xi_1, \xi_2\in X_{M(T),T}$, where $X_{M(T),T}$ is defined in
 \eqref{eq:2-3-9}, and thus,  we have the same estimates of
 \eqref{eq:2-4-1} and \eqref{eq:2-4-2}, namely we have,
 \begin{equation}
  \label{eq:3-0-4}
  \left\|
   \frac{(b(x,t))^2}{2D(x)}
   f^{\mathrm{eq}}(x) (|\nabla \xi_1|^2-|\nabla
   \xi_2|^2)
   \right\|_{C^{\alpha,\alpha/2}(\Omega\times[0,T))}
   \leq
   \Cr{const:2.5}M(T)
   (T^{(1+\alpha)/2}+T^{1/2})^2
   \|
   \xi_1-\xi_2
   \|_{C^{2+\alpha,1+\alpha/2}(\Omega\times[0,T))},
 \end{equation}
 and
 \begin{equation}
  \label{eq:3-0-5}
  \begin{split}
   &\quad
   \left\|
   \frac{(b(x,t))^2}{2(D(x))^2}
   f^{\mathrm{eq}}(x)
   (
   \xi_1
   \nabla \xi_1
   -
   \xi_2
   \nabla \xi_2
   )
   \cdot
   \nabla
   D(x)
   \right\|_{C^{\alpha,\alpha/2}(\Omega\times[0,T))} \\
   &\leq
   \Cr{const:2.6}M(T)
   (T^{(1+\alpha)/2}+T^{1/2})(T+T^{1-\alpha/2})
   \|
   \xi_1-\xi_2
   \|_{C^{2+\alpha,1+\alpha/2}(\Omega\times[0,T))},
  \end{split}
 \end{equation}
 where constants,
 \begin{equation}
  \Cr{const:2.5}
   =
   9
   \left\|
    \frac{(b(x,t))^2}{D(x)}
    f^{\mathrm{eq}}(x)
   \right\|_{C^{\alpha,\alpha/2}(\Omega\times[0,T))}, \mbox{ and }
   \quad
   \Cr{const:2.6}
   =
   9
   \left\|
    \frac{(b(x,t))^2}{(D(x))^2}
    f^{\mathrm{eq}}(x)
    \nabla D(x)
   \right\|_{C^{\alpha,\alpha/2}(\Omega\times[0,T))}.
 \end{equation} 
 Combining \eqref{eq:3-0-1}, \eqref{eq:3-0-4} and \eqref{eq:3-0-5}, we
 obtain the estimate,
 \begin{equation}
  \label{eq:3-0-8}
  \|
   G(\xi_1)-G(\xi_2)
   \|_{C^{\alpha,\alpha/2}(\Omega\times[0,T))}
   \leq
   \Cl{const:3.1}
   M(T)\kappa(T)
   \|\xi_1-\xi_2\|_{C^{2+\alpha,1+\alpha/2}(\Omega\times[0,T))},
 \end{equation}
 where $\Cr{const:3.1}=\max\{\Cr{const:2.5},\ \Cr{const:2.6}\}>0$ and,
 \begin{equation}
  \kappa(T)
   =
   (T^{(1+\alpha)/2}+T^{1/2})^2
   +
   (T^{(1+\alpha)/2}+T^{1/2})
   (T+T^{1-\alpha/2}).
 \end{equation}
 Note that $M(T)$ and $\kappa(T)$ are increasing with respect to $T>0$, and
 $\kappa(T)\rightarrow0$ as $T\downarrow0$. Therefore, take $T>0$ such that,
 \begin{equation}
  \label{eq:3-0-7}
   \Cr{const:2.2}
   \Cr{const:3.1}  M(T)\kappa(T)<1.
 \end{equation}
 Then combining \eqref{eq:3-0-6}, \eqref{eq:3-0-8}, and \eqref{eq:3-0-7},
 we obtain that,
 \begin{equation}
  \begin{split}
   \|\xi_1-\xi_2\|_{C^{2+\alpha,1+\alpha/2}(\Omega\times[0,T))}
   &\leq
   \Cr{const:2.2}
   \Cr{const:3.1}  M(T)\kappa(T)
   \|
   \xi_1-\xi_2
   \|_{C^{2+\alpha,1+\alpha/2}(\Omega\times[0,T))} \\  
   &<
   \|
   \xi_1-\xi_2
   \|_{C^{2+\alpha,1+\alpha/2}(\Omega\times[0,T))},  
  \end{split} 
 \end{equation}
 which is a contradiction. Thus, we established that $\xi_1=\xi_2$ in
 $\Omega\times[0,T)$.
\end{proof}

\section{Conclusion}\label{sec:4}
\par In this paper,  we presented a new nonlinear Fokker-Planck equation
which satisfies a special energy law with the inhomogeneous absolute
 temperature of the system. Such models emerge as a part of grain
growth modeling in polycrystalline materials. We showed local
existence and uniqueness of the solution of the Fokker-Planck
system. Large time asymptotic analysis of the proposed Fokker-Planck
model, as well as numerical simulations of the system will
be presented in a forthcoming paper
\cite{Katya-Kamala-Chun-Masashi}. As a part of our future research, we
will further extend such Fokker-Planck systems to the modeling of the
evolution of the
grain boundary network that undergoes disappearance/critical events, e.g. \cite{epshteyn2021stochastic,Katya-Chun-Mzn4}.

\section*{Acknowledgments}
Yekaterina
Epshteyn acknowledges partial support of NSF DMS-1905463 and of NSF DMS-2118172,  Masashi Mizuno
acknowledges partial support of JSPS KAKENHI Grant No. JP18K13446 and Chun Liu acknowledges partial support of
NSF DMS-1950868 and NSF DMS-2118181.
\bibliographystyle{plain}
\bibliography{references}

\end{document}